\documentclass[11pt,twoside]{amsart}

\usepackage{setspace, a4wide}
\usepackage[T1]{fontenc}
\usepackage[english]{babel}
\usepackage{amssymb,amsfonts,amsthm,amsmath,mathrsfs}
\usepackage{layout,version,enumerate}
\usepackage[all,matrix]{xy}
\usepackage[varg]{pxfonts}
\usepackage{indentfirst}
\usepackage{url}
\usepackage{color}
\usepackage[colorlinks=true,linkcolor=blue,citecolor=magenta, pagebackref]{hyperref}
\usepackage[numbers]{natbib}
\usepackage{tikz-cd} 

\theoremstyle{plain}
\newtheorem{thm}{Theorem}[section]
\newtheorem{prop}[thm]{Proposition}
\newtheorem{crl}[thm]{Corollary}
\newtheorem{lem}[thm]{Lemma}
\newtheorem*{thm*}{Theorem}
\newtheorem*{fact*}{Fact}

\theoremstyle{definition}
\newtheorem{dfn}[thm]{Definition}
\newtheorem*{dfn*}{Definition}
\newtheorem*{app}{Approximation problem for $\cp$}
\newtheorem*{lapp}{Lattice approximation problem for $\cp$}

\theoremstyle{remark}
\newtheorem{rmk}[thm]{Remark}

\newcommand{\bn}{\mathbb{N}}
\newcommand{\bz}{\mathbb{Z}}
\newcommand{\br}{\mathbb{R}}

\newcommand{\rb}{\mathscr B}
\newcommand{\ca}{\mathcal A}

\newcommand{\cg}{\mathcal G}

\newcommand{\ul}{\mathfrak u}
\newcommand{\rel}{\mathcal R}

\newcommand{\co}{\mathcal O}
\newcommand{\cp}{\mathcal P}

\newcommand{\ph}{\varphi}
\newcommand{\eps}{\varepsilon}
\newcommand{\id}{\mathfrak{id}}

\newcommand{\rank}{\mathit{d}}

\renewcommand{\id}{\mathrm{1}}
\newcommand{\haar}{\lambda}
\def\defin#1{\textbf{\emph{#1}}}  
\newcommand{\pmp}{p.m.p.\ }

\usepackage{tikz}
\usetikzlibrary{matrix,arrows}

\renewcommand{\part}{\operatorname{Part}}

\newcommand{\aut}{\operatorname{Aut}}
\newcommand{\lcsc}{l.c.s.c.\ }

\newcommand{\binte}{\overline{\operatorname{Int}}}
\newcommand{\covol}{\operatorname{covol}}
\newcommand{\vol}{\operatorname{vol}}

\newcommand{\cost}{\mathrm{Cost}}
\newcommand{\SmallHack}[1]{\lowercase{#1}}
\newcommand{\Bibkeyhack}[1]{}

\begin{document}

\onehalfspace

\thispagestyle{empty}
\title{Asymptotic invariants of lattices in locally compact groups}

\author{Alessandro Carderi}
\address{A.C., Institut für Algebra und Geometrie, KIT, 76128 Karlsruhe, Germany}
\email{alessandro.carderi@kit.edu}


\begin{abstract}
{ The aim of this work is to understand some of the asymptotic properties of sequences of lattices in a fixed locally compact group. In particular we will study the asymptotic growth of the Betti numbers of the lattices renormalized by the covolume and the rank gradient, the minimal number of generators also renormalized by the covolume. For doing so we will consider the ultraproduct of the sequence of actions of the locally compact group on the coset spaces and we will show how the properties of one of its cross sections are related to the asymptotic properties of the lattices.
}
\end{abstract}

\maketitle

\tableofcontents
\addtocontents{toc}{\protect\setcounter{tocdepth}{-3}} 
\section*{Introduction}

Let us consider a real-valued invariant $\cp$ of countable groups, that is, $\cp$ assigns to every countable group $\Gamma$ a (possibly infinite) real number $\cp(\Gamma)$. The main examples for us will be when $\cp=\rank$ is the rank, the minimal number of generators, and when $\cp=b_i$ is the $i$-th Betti number, that is, the dimension of the $i$-th homology group with rational coefficients. However the questions that we address are still interesting and often open for many other invariants such as the rank of the abelianization or the dimension of the homology with coefficients in a finite field. We want to understand the following problem, see \cite{luc2016}.

\begin{app} 
For a countable, residually finite group $\Gamma$ and a sequence of finite index subgroups $\{\Gamma_n\}_n$ of $\Gamma$ such that the index $[\Gamma\colon\Gamma_n]$ tends to infinity does the sequence $\cp(\Gamma_n)/[\Gamma\colon\Gamma_n]$ converge? If so, does the limit depend on the sequence $\{\Gamma_n\}_n$? Can we compute this limit?
\end{app}

A first striking result in this direction is the L\"uck's approximation theorem \cite{luc1994} which states that whenever $\Gamma$ is the fundamental group of a compact manifold and $\{\Gamma_n\}_n$ is a (nested) chain of normal subgroups of $\Gamma$ satisfying $\cap_n \Gamma_n=\{\id_\Gamma\}$, then the sequence $b_i(\Gamma_n)/[\Gamma\colon\Gamma_n]$ converges to the $i$-th $\ell^2$-Betti number of the group $\Gamma$, which will be denoted by $\beta_i(\Gamma)$. Therefore, under some assumptions on the group and on the chain, the approximation problem is completely understood and the limit does not depend on the chain. The needed assumptions on the sequence of finite index subgroups were later weakened by Farber.

\begin{dfn*}
  A sequence of finite index subgroups $\{\Gamma_n\}_n$ is called \textit{Farber} if for every $\gamma\in\Gamma\setminus\{\id_\Gamma\}$ we have \[\lim_n\frac{\left|\left\{g\Gamma_n\in \Gamma/\Gamma_n\colon \gamma\in g\Gamma_n g^{-1}\right\}\right|}{[\Gamma\colon\Gamma_n]}=0.\]
\end{dfn*}
 
Farber proved in \cite{far1998} that the sequence $b_i(\Gamma_n)/[\Gamma\colon\Gamma_n]$ also converges to $\beta_i(\Gamma)$ whenever $\{\Gamma_n\}_n$ is a Farber and nested chain of finite index subgroups of $\Gamma$. The approximation problem for $b_i$ was later studied in several different occasion, for example in the case of non Farber chains in \cite{ber2004}, for sofic approximations in \cite{ele2005} and in a more general setting in \cite{CGS}.

Another invariant for which the approximation problem has been widely studied is the rank. Ab\'ert and Nikolov proved in \cite{abe2012} that whenever $\Gamma$ is a finitely generated group and $\{\Gamma_n\}_n$ is a Farber chain of finite index subgroups, then the sequence $\rank(\Gamma_n)/[\Gamma\colon\Gamma_n]$ converges towards the cost of the associated profinite action (minus $1$). In this case however it is still unknown whether the limit depends on the chosen Farber sequence, indeed it is unknown whether all free \pmp actions of a fixed countable group have the same cost, \cite{gab2000}. Ab\'ert and Nikolov's theorem was later generalized to non-nested sequences in \cite{tot2017} and \cite{CGS}, where in general only one inequality is proved.

In our work we are interested in a variant of the approximation problem for lattices in locally compact groups.

\begin{lapp} 
For a locally compact, second countable group $G$ and a sequence of lattices $\{\Gamma_n\}_n$ of $G$ such that the covolume $\covol(\Gamma_n)$ tends to infinity does the sequence $\cp(\Gamma_n)/\covol(\Gamma_n)$ converge? If so, does the limit depend on the sequence $\{\Gamma_n\}_n$? Can we compute this limit?
\end{lapp}

The lattice approximation problem was first studied for Betti numbers in \cite{abe2017} where, among many other things, it is proven that if $G$ is a connected center-free semi-simple Lie group and $\{\Gamma_n\}_n$ is a sequence of lattices which is \textit{nowhere thin} and \textit{almost everywhere thick}, then the sequence $b_i(\Gamma_n)/\covol(\Gamma_n)$ converges to the $i$-th $\ell^2$-Betti number of the associated homogeneous space. 

\begin{dfn*}
  A sequence of lattices $\{\Gamma_n\}_n$ is
  \begin{itemize}
  \item \textit{nowhere thin} (or \textit{uniformly discrete}) if there exists a neighborhood of the identity $U\subseteq G$ such that for every $g\in G$ and $n\in\bn$ we have that $g\Gamma_ng^{-1}\cap U=\{\id_G\}$;
  \item \textit{almost everywhere thick} (or \textit{Farber}) if for every neighborhood of the identity $U\subseteq G$ we have \[\lim_n\frac{\vol\left(\left\{ g\Gamma_n\in G/\Gamma_n\colon g\Gamma_ng^{-1}\cap U=\{\id_G\}\right\}\right)}{\covol(\Gamma_n)}=1.\]
  \end{itemize}
\end{dfn*}

A similar result was obtained in \cite{pet2016} for totally disconnected groups acting cocompactly on simplicial complexes. In this context they were also able to get an inequality for the limit without the ``nowhere thinness'' assumption. One of the purposes of this work is to give a unified proof of these two results and to generalize them in the context of a general locally compact group.

\begin{thm*}[Corollary \ref{crl:excomp}]\label{thmI:l2}
  Let $G$ be a unimodular \lcsc compactly generated group and let us fix a Haar measure $\haar$ on $G$. Let $\{\Gamma_n\}_n$ be a Farber sequence of cocompact torsion free lattices of $G$. Assume that one of the following conditions holds. 
  \begin{itemize}
  \item The group $G$ is almost connected.
  \item The group $G$ is totally disconnected and it acts cocompactly and properly on a contractible simplicial complex.
  \item The quotient $G/G_0$ satisfies the previous condition, where $G_0\leq G$ denotes the connected component of the identity
  \end{itemize}
  Then \[\beta_i(G)\leq \liminf_n \frac{b_i(\Gamma_n)}{\covol(\Gamma_n)}\] and if the sequence $\{\Gamma_n\}_n$ is nowhere thin, then \[\beta_i(G)= \lim_n \frac{b_i(\Gamma_n)}{\covol(\Gamma_n)}.\] 
\end{thm*}

Here by $\beta_i(G)$ we mean the $i$-th $\ell^2$-Betti number of the locally compact group $G$ as defined in \cite{pet2013} and \cite{kye2015}. In particular we derive as a corollary the above stated theorem in \cite{abe2017} and in \cite{pet2016} for torsion free lattices\footnote{The general form of Theorem $1.3$ of \cite{pet2016} can also be derived from our techniques but in this text we will only work with torsion free lattices.}. But our result can be also applied to a large new class of locally compact groups with interesting lattices, the S-adic groups, the groups which appear as (local) products of real and p-adic Lie groups.

We will prove the above theorem in a much more general context, see Theorem \ref{thm:mainl2}, of actions on what we call well-covered $G$-spaces, see Definition \ref{dfn:wellcovered}. In particular a similar convergence will be derived for every cocompact, proper and isometric action on a contractible Riemannian manifold or for every cocompact, proper and simplicial action on a contractible simplicial space. The theorem above will then be a corollary of Theorem \ref{thm:mainl2} by considering the action of $G$ on a contractible well-covered $G$-space.

Let us observe that for some groups $G$ which either have property $(T)$ or $(\tau)$ and which satisfy the conclusion of the Stuck-Zimmer theorem, any sequence of lattices whose covolume tends to infinity is automatically almost everywhere thick, or Farber see \cite{abe2017}, \cite{gel2017} and \cite{lev2017a}. We will dedicate a small note about this phenomenon in a further work \cite{carta}.

Using similar techniques we are also able to obtain an analogous theorem for the lattice approximation problem for the rank. For this we need the notion of cost for a \pmp action of a locally compact group. Such a cost can be defined, see Definition \ref{dfn:costlc}, to be the cost of one its cross sections renormalized by the covolume, that is, if $G$ acts on $(X,\mu)$ and $Y\subseteq X$ is a cross section, then \[\cost\left(G\curvearrowright X\right)\coloneqq 1+\frac{\cost(\rel_Y)-1}{\covol(Y)}.\] 

\begin{thm*}[Theorem \ref{thm:costgen}]
  Let $G$ be a compactly generated unimodular \lcsc group and fix a Haar measure $\haar$. Let $\{\Gamma_n\}_n$ be a nowhere thin and almost everywhere thick sequence of lattices of $G$. Then \[\cost\left(G\curvearrowright [G/\Gamma_n]_\ul^R\right)\geq 1+\lim_\ul\frac{\rank(\Gamma_n)-1}{\covol(\Gamma_n)}\geq \cost(G).\]
\end{thm*}

The probability space $[G/\Gamma_n]_\ul^R$ is a measure-theoretic ultraproduct of the measure spaces $\{G/\Gamma_n\}_n$ and will be the main object of study in this work. The cost of this action is in most of the cases unknown, it seems very hard to compute and we will not give concrete bounds of it. 

\subsection*{The case of Lie groups.} 
Loosely speaking the assumption of nowhere thinness is needed because the homology/cost of the limit will only depend on the thick part of the quotients by the lattices $\Gamma_n$. For Lie groups we can make this observation a little bit more precise.

Let $G$ be a connected Lie group, let $K_0<G$ be a maximal compact subgroup and consider the homogeneous space $M\coloneqq K_0\backslash G$. Let us fix a Riemannian metric on $M$ for which the action of $G$ is isometric and we denote by $d_M$ the associated (geodesic) distance. For a positive real number $\eps$ we consider the $\eps$-thick part \[(M/\Gamma)_\eps\coloneqq\left\{p\in M\colon d_M(p,p\gamma)\geq \eps\ \forall \gamma\in\Gamma\setminus \{\id\}\right\}/\Gamma.\]

Finally if $P\subseteq Q$ is an inclusion of topological spaces we will denote by $\nabla_i(P,Q)$ the dimension of the
image of the $i$-th homology group of $P$ inside the homology of $Q$. The following theorem follows from Theorem \ref{thm:epsneigh} and Lemma \ref{lem: riemann well-covered}.

\begin{thm*}
  Let $G$ be a unimodular \lcsc compactly generated group and let us fix a Haar measure $\haar$ on $G$. Let $M$ be a contractible Riemannian manifold equipped with the Riemannian metric $d_M$ and on which $G$ acts by isometries.  
  Let $\{\Gamma_n\}_n$ be a Farber sequence of torsion free lattices of $\Gamma$. Then for every $\eps>0$ we have that \[\beta_i(G)=\lim_n \frac{\nabla_i((M/\Gamma_n)_\eps,M/\Gamma_n)}{\covol(\Gamma_n)}.\]
\end{thm*}

The theorem becomes easier to understand whenever we have that $(M/\Gamma_n)_\eps$ already contains all the homological information of the space, that is, when $\nabla_i((M/\Gamma_n)_\eps,M/\Gamma_n)=b_i(M/\Gamma_n)$. In this work we will not study this phenomenon. We will just observe in Corollary \ref{crl:gela} that an easy consequence of the thin-thick decomposition, \cite[Theorem 4.5.6]{thu1997}, the work of Gromov, \cite{gro1985} and the more recent work of Gelander, \cite{gel2004}, imply that it is the case for $i=1$, for higher rank non cocompact lattices and in rank $1$. In this way we are able to get a result similar to Theorem 1.8 of \cite{abe2017} and to a recent work by Ab\'ert, Bergeron, Biringer and Gelander \cite{abbg2018}.

Similarly, for the rank-cost approximation problem we obtain the following. 

\begin{thm*}[Theorem \ref{thm:costgela}]
  Let $G$ be a semisimple Lie group without compact factors and fix a Haar measure on $G$. Let $\{\Gamma_n\}_n$ be a Farber sequence of torsion free lattices. Then \[\cost\left(G\curvearrowright [G/\Gamma_n]_\ul^{R}\right)\geq 1+\lim_\ul \frac{\rank(\Gamma_n)-1}{\covol(\Gamma_n)}\geq \cost(G).\]
\end{thm*}

\subsection*{Ultraproducts of actions}
All the above stated theorems will be derived using the notion of ultraproducts of \pmp actions of locally compact groups along some fixed ultrafilter $\ul$. Note that if a locally compact group $G$ acts measurably on a sequence of probability spaces, then it will not act measurably on the associated ultraproduct. But it will however act measurably on a factor of it, which in the context of von Neumann algebras is sometimes called the equicontinuous part, \cite{tom2017}. We will use a similar ultraproduct which we will call the \textit{regular ultraproduct}. 

\begin{dfn*}[Definition \ref{dfn:regset} and \ref{dfn:regseq}]
  Assume that the \lcsc group $G$ acts measurably on the sequence of probability spaces $\{(X_n,\mu_n)\}_n$. We define
$[X_n]_\ul^R\coloneqq\prod_{n\in \bn}X_n/\sim^R_\ul$ for the equivalence relation defined by $(x_n)_n \sim^R_\ul (y_n)_n$ if there is a sequence $g_n\in G$ such that $\lim_\ul g_n=\id_G$ and $g_nx_{n}= y_{n}$ for $\ul$-almost every $n$.

We say that a sequence of subsets $\{A_n\subseteq X_n\}_n$ is ($\ul$-)\textit{regular} if for every $\eps>0$ there exists a neighborhood of the identity $U\subseteq G$ such that $UA_n$ is measurable and such that $\mu_n(UA_n\setminus A_n)\leq \eps$ for $\ul$-almost every $n$.
\end{dfn*}

The \textit{regular ultraproduct} is the probability space whose underlying set is $[X_n]_\ul^R$, the $\sigma$-algebra of measurable subsets is generated by the class of regular sequences $[A_n]_\ul^R$ and for such a sequence the measure is defined to be $\mu_\ul^R([A_n]_\ul^R)\coloneqq\lim_\ul\mu_n(A_n)$.

\begin{thm*}[Theorem \ref{thm:ultrabasic}]
  Let $G$ be a \lcsc group. Suppose that $G$ acts in a Borel manner on the sequence of standard probability spaces $(X_n,\mu_n)$ preserving the measure. Then $G$ acts continuously and measurably on the regular ultraproduct $([X_n]_\ul^R,\rb_\ul^R,\mu_\ul^R)$. 
\end{thm*}

One of the main tools for understanding \pmp actions of locally compact groups is the notion of cross section. For the regular ultraproduct we have an explicit one: the cross section of the regular ultraproduct is the ultraproduct of the cross sections. 

\begin{thm*}[Theorem \ref{thm:maincross}]
  Let $G$ be a \lcsc unimodular group, fix an open neighborhood of the identity $U'\subseteq G$ and a Haar measure $\haar$. Suppose that $G$ acts on the sequence of standard probability spaces $(X_n,\mu_n)$. Assume that for every $n$ there exists a $U'$-cross section $Y_n\subseteq X_n$ such that $\{\covol(Y_n)^{-1}\}_n$ is bounded and denote by $\nu_n$ the induced measure on $Y_n$. Take a compact neighborhood of the identity $U\subseteq U'$. Let $Y_\ul$ be the ultraproduct of $\{(Y_n,\nu_n)\}_n$. Consider the function
\begin{align*}
  \Psi\colon G\times Y_\ul\rightarrow X_\ul^R\quad\text{  }\quad\Psi(g,[y_n]_\ul)\coloneqq [\Phi_n(g,y_n)]_\ul^{R}.
\end{align*}
Then $(Y_\ul,\Psi)$ is an external $U$-cross section of the action of $G$ on the $G$-invariant measurable subset $G[\Phi_n(U\times Y_n)]_\ul^R\subseteq  X_\ul^{R}$ with covolume $\covol(Y_\ul)=\lim_\ul\covol(Y_n)$.
\end{thm*}

Moreover we will show that this property extends also naturally to the associated cross equivalence relations, see Theorem \ref{thm:mainrel}.

\subsection*{Sketch of the proof(s)}

Our understanding of the lattice approximation problem will be derived from the study of the regular ultraproduct and its cross section. Indeed using that the cross section of the regular ultraproduct is the ultraproduct of the cross sections we will reduce the problem to the case of ultraproducts of finite graphed equivalence relations studied in \cite{CGS}. The idea behind the proof is not hard but since we will have to work with non standard probability spaces we have to be very careful about the measurability problems. We will now sketch the proof of one of these theorems, say Theorem \ref{thm:epsneigh}, assuming that \textit{everything is measurable}.

Let us fix a connected Lie group $G$, denote by $K_0$ a maximal compact subgroup. Put $M\coloneqq K_0\backslash G$ and let $d_M$ be a geodesic $G$-invariant metric on $M$. Consider an almost everywhere thick sequence of lattices $\{\Gamma_n\}_n$ and a positive real number $\eps$. Take a sufficiently small $\eta$ and for every $n$ consider a subset $D_n\subseteq (M/\Gamma_n)_\eps$ which is maximal $\eta$-discrete so that the cover by the ball or radius $2\eta$ and center in $D_n$ covers $(M/\Gamma_n)_\eps$. We will now assume that the union of these balls is $(M/\Gamma_n)_\eps$, which is obviously not true in general but $(M/\Gamma_n)_\eps$ can be easily sandwiched between two such covers.

Let $\widetilde D_n$ be the associated $\eta$-discrete $\Gamma_n$-invariant subset of $M$. For every $g\in G$ we consider $D_ng$ to be the translate of $D_n$ inside $M/(g^{-1}\Gamma_n g)$ and observe that $D_ng$ is the quotient of $\widetilde D_n g$ by $g^{-1}\Gamma_n g$. In this way we get a random pointed $\eta$-discrete subset of $M$ indexed by $G/\Gamma_n$ (up to an inversion) which can also be thought as a random pointed $\eta$-discrete subset of the ``random space'' $M/(g^{-1}\Gamma_ng)$. Let us denote by $E_n$ a random element according to this measure, that is, a random translate of $\widetilde D_n$.

Let us fix an ultrafilter $\ul$. For every sequence of pointed $\eta$-discrete subsets $\{E_n\}_n$ we can form the pointed (at the identity) Hausdorff limit to obtain a $\eta$-discrete subset $E_\ul$ of $M$. Now the assumption that $\{\Gamma_n\}_n$ is almost everywhere thick translates to the fact that $E_\ul$ is $2\eta$-dense. Therefore if we consider the space of sequences of elements in $G/\Gamma_n$ modulo the fact the Hausdorff limit are the same, that is, the \textit{regular ultraproduct}, we obtain a random $G$-invariant $\eta$-discrete, $2\eta$-dense subset of $M$. By considering the balls of radius $2\eta$ around the points of $E_\ul$, we obtain a random cover of $M$. Let us discretize the problem.

For every $n$ by construction the set $D_n$ is finite. We consider now the finite set of translates of $D_n$ which are pointed at an element of the original $D_n$. When we take the nerve of the associated covering by balls of radius $2\eta$ we are just considering a finite simplicial complex and we are choosing its root at random. 

Now as above we let $n$ go to infinity and we take the ultraproduct of the sequence of pointed finite simplicial complexes as in \cite{CGS} (which is closely related to the Benjamini-Schramm convergence of finite simplicial complexes). This will give us a random (not anymore finite) simplicial complex. Again, the hypothesis that the sequence is Farber, will tell us that this random simplicial complex is contractible. Now the fact that the ultraproduct of the cross sections is the cross section of the regular ultraproduct will tell us that the random simplicial complex obtained in this way is a field of simplicial complexes of the \textit{cross equivalence relation} of an action of $G$. 

We now have to observe two facts. The first is that \cite{kye2015} implies that the $\ell^2$-Betti numbers of the discretization and the process itself are the same (modulo a covolume factor). The second is that it follows from \cite{CGS} that the $\ell^2$-Betti numbers of the finite random simplicial complexes converge as in L\"uck approximation theorem to the $\ell^2$-Betti numbers of the ultraproduct of the sequence. Finally a little bit of checking on covolumes and constants will give the desired result.

\section*{Acknowledgements}

This research was supported by the ERC Consolidator Grant No. 681207 and by the Deutsche Forschungsgemeinschaft  (German Research Foundation, 281869850 —RTG 2229). The author is grateful to Andreas Thom and Vadim Alekseev for countless interesting discussions over a Schweinshaxe or a cup of tea. This work could not have been completed without their help and support. The author thanks the referees for their useful comments and remarks.

\addtocontents{toc}{\protect\setcounter{tocdepth}{1}} 
\section{Ultraproducts of actions of locally compact groups}

One of the main concepts in this work will be the notion of \textit{ultraproduct}. Hence we fix now and for the entire work a non-principal ultrafilter $\ul$ over the natural numbers $\bn$. We will say that a property of a sequence $\{a_n\}_n$ holds for $\ul$-\textit{almost every} $n$ if the set of $n\in\bn$ for which $a_n$ satisfy the property is in the ultrafilter $\ul$.

In this text measure spaces will always be \textit{complete} and $\sigma$-\textit{finite}, in most of the cases \textit{probability spaces}. We will denote by $(X,\rb,\mu)$ a measure space, or briefly $(X,\mu)$ when we do not need to specify the $\sigma$-algebra $\rb$. We say that the $\sigma$-finite measure space $(X,\rb,\mu)$ is \textit{separable} if the $\sigma$-algebra $\rb$ is generated up to null-sets by a countable subset. A measurable map between measure spaces is a map $T\colon (X,\mu)\rightarrow (Z,\nu)$ such that for every measurable subset $A\subseteq Z$ we have that $T^{-1}(A)$ is measurable. We will say that $T$ is (inverse) \textit{measure preserving} if $\mu(T^{-1}(A))=\nu(A)$ for every measurable subset $A\subseteq Z$. 

We will work with \textit{locally compact second countable} (l.c.s.c.) groups, often denoted by $G$ and almost always unimodular and compactly generated. We will denote their left Haar measure by $\haar$. By a \textit{neighborhood of the identity} we mean a precompact Borel subset of $G$ which contains an open subset containing the identity, which is usually denoted by $\id_G$. In particular for us neighborhoods of the identity have finite Haar measure. A \textit{measure preserving action} of $G$ on a measure space $(X,\mu)$ is an action of $G$ on the set $X$ such that for every $g\in G$ the map $x\mapsto gx$ is a measure preserving transformation of the space. Note that $gx$ is defined for every $x\in X$. For an action of $G$ on the probability space $(X,\mu)$ we will denote by $\Phi\colon G\times X\rightarrow X$ the action map, that is, $\Phi(g,x)\coloneqq gx$. 

\begin{dfn}\label{dfn:measurability}
We say that an action is \defin{measurable} if the action map $\Phi\colon G\times X\rightarrow X$ is measurable.
\end{dfn}

As an example any Borel action of a \lcsc group on a standard Borel space is measurable. Indeed if $A\subseteq X$ is measurable, then there are Borel subsets $A_1$ and $A_2$ such that $A_1\subseteq A\subseteq A_2$ and $\mu(A_1)=\mu(A_2)$. Clearly $\Phi^{-1}(A_1)\subseteq \Phi^{-1}(A)\subseteq \Phi^{-1}(A_2)$ and by assumption $\Phi^{-1}(A_i)$ is Borel. Therefore Fubini\footnote{which holds for every probability space, see \cite[Theorem 252B]{fremlin}} implies that $\haar\times \mu(\Phi^{-1}(A_2\setminus A_1))=0$ and hence that $\Phi^{-1}(A)$ is measurable. 

We will say that a measurable action of $G$ on the measure space $(X,\mu)$ is \textit{probability measure preserving} (p.m.p.) if $\mu$ is a probability measure and the action preserves it.

\subsection{Continuity of actions}
Denote by $\aut(X,\mu)$ the group of all measure preserving automorphism of the probability space $(X,\mu)$. This group is equipped with a topology, called the \textit{weak topology}, which is induced by the functions $T\mapsto \mu(A\Delta TA)$ for all measurable subsets $A\subseteq X$. This topology is Polish whenever $(X,\mu)$ is a standard probability space. 

An action of a \lcsc group $G$ on a probability space $(X,\mu)$ is \defin{continuous} if the induced map $\rho\colon G\rightarrow \aut(X,\mu)$ is continuous. 

Since we will encounter non-continuous actions of locally compact groups, it is convenient to introduce some more terminology. Assume that the \lcsc group $G$ acts on the probability measure space $(X,\mu)$ preserving the measure. Let $U\subseteq G$ be a neighborhood of the identity and take $\eps>0$. A measurable subset $A\subseteq X$ is $(U,\eps)$-\defin{invariant} if $\sup_{g\in U}\mu(A\Delta gA)\leq \eps$.

Observe that since the action is measure preserving, $A$ is $(U,\eps)$-invariant if and only if it is $(U\cup U^{-1},\eps)$-invariant. We say that a measure preserving action of $G$ on the probability space $(X,\mu)$ is \textit{continuous at the measurable subset} $A\subseteq X$ if for every $\eps>0$ there exists $U$ such that $A$ is $(U,\eps)$-invariant. Clearly by definition of the weak topology on $\aut(X,\mu)$, an action is continuous if and only if it is continuous at every element of finite measure. The following lemma is a straightforward computation. 

\begin{lem}\label{lem:algop}
  Assume that the \lcsc group $G$ acts on the probability space $(X,\mu)$ and consider measurable subsets $A,B\subseteq X$ which are $(U,\eps)$-invariant. Then $gA$ is $(gUg^{-1},\eps)$-invariant, $X\setminus A$ is $(U,\eps)$-invariant and $A\cap B$ and $A\cup B$ are $(U,2\eps)$-invariant. In particular if $G$ is continuous at $A$ and $B$, then it is continuous at $gA$, $A\cap B$ and $X\setminus A$. 
\end{lem}

If a \lcsc group $G$ acts in a Borel manner on the standard probability space $(X,\mu)$, then the induced map from $G$ to $\aut(X,\mu)$ is Borel and hence is automatically continuous, see Theorem 2.3.3 of \cite{gao}. Therefore any Borel action is automatically continuous. Let us sketch the proof of the analogous statement for general measure spaces.

\begin{prop}\label{prop:automacont}
  Assume that the \lcsc group $G$ acts measurably on the probability space $(X,\mu)$. Then the action is continuous.
\end{prop}
\begin{proof}
  Take a measurable subset $A\subseteq X$, let $U\subseteq G$ be a symmetric neighborhood of the identity and denote by $\lambda$ a Haar measure on $G$. Take $\eps\in (0,\lambda(U))$. Since $\Phi^{-1}(A)\subseteq G\times X$ is measurable there exists a measurable finite partition $\{W_i\}_i$ of $U$ and measurable subsets $B_i\subseteq X$ such that if we set $T\coloneqq \cup_i W_i\times B_i$, then $\lambda\times\mu(T\Delta(U\times X\cap \Phi^{-1}(A)))<\eps^2$. For $g\in G$, set $T_g\coloneqq \{x\in X\colon (g,x)\in T\}$ and observe that if $g,h\in W_i$, then $T_g=T_h=B_i$. Put $V\coloneqq \left\{g\in U\colon \mu(gA\Delta T_g)\leq \eps \right\}$. Observe that $\lambda\times \mu(U\setminus V)\leq \eps$. 

Fix $i$ such that $Z\coloneqq W_i\cap V$ has positive Haar measure. By \cite[Proposition 443D (c)]{fremlin} whenever $E\subseteq G$ has positive Haar measure, the set $E^{-1}E$ is a neighborhood of the identity. In particular, $Z^{-1}Z$ contains an open neighborhood of the identity $U'\subseteq U$. For every $u\in U'$ there are $g,h\in Z$ such that $u=g^{-1}h$. Therefore \[\mu(uA\Delta A)=\mu(gA\Delta hA)\leq \mu(gA\Delta T_g)+\mu(hA\Delta T_h)\leq 2\eps.\qedhere\]
\end{proof}

Since the groups we consider are separable every continuous action is \textit{locally} described by a separable $\sigma$-algebra.

\begin{lem}\label{lem:techcont}
  Assume that $G$ acts on the probability space $(X,\rb,\mu)$ preserving the measure. Suppose that $P$ is a subset of $\rb$ such that $G$ is continuous at every element of $P$. Then $G$ acts continuously on the $G$-invariant $\sigma$-subalgebra $\rb_P$ generated by $P$. 

 If $P$ can be generated by a countable subset, then $\rb_P$ is separable. 
\end{lem}
\begin{proof} 
  Let $H\leq G$ be a dense countable subgroup and let us denote by $\ca_H$ the algebra of subsets generated by $\{hP\}_{h\in H}$. By Lemma \ref{lem:algop}, the group $G$ is continuous at every element of $\ca_H$. Observe that if $P$ is generated by a countable subset, then $\ca_H$ is. Since $H\leq G$ is dense and the action is continuous, we have that $\ca_H$ is dense in the algebra $\ca_G$ generated by the $G$-translates of $P$. Remark that $\ca_H\subseteq \rb_P$ is dense, that is for every element $B\in\rb_P$ and $\eps>0$ there exists $A_\eps\in\ca_H$ such that $\mu(B\Delta A_\eps)<\eps$, see for example \cite[Proposition 136H]{fremlin}. Finally, let us check that the action of $G$ is continous at every element of $\rb_P$. Given $B\in \rb_P$ and $\eps>0$, there is an element $A_\eps\in\ca_H$ and a neighborhood of the identity $U\subseteq G$ such that $A_\eps$ is $(U,\eps)$-invariant and $\mu(A_\eps\Delta B)\leq \eps$. Then $B$ is $(U,3\eps)$-invariant, indeed \[\sup_{g\in U}\mu(B\Delta gB)\leq \sup_{g\in U}\left(\mu(A_\eps\Delta B)+\mu(A_\eps\Delta gA_\eps)+\mu(gB\Delta gA_\eps)\right)\leq 3\eps.\qedhere\]
\end{proof}

 Suppose that $G$ acts on the probability space $(X,\rb,\mu)$ preserving the measure. The set of sub-$\sigma$-algebras of $\rb$ for which the action of $G$ is continuous admits joins, that is for every two continuous sub-$\sigma$-algebras the $\sigma$-algebra generated by them is still continuous. Therefore there exists a maximal sub-$\sigma$-algebra of $\rb$ for which the action of $G$ is continuous. This $\sigma$-algebra is explicit: it is the $\sigma$-algebra generated by all the $A\in \rb$ such that $G$ is continuous at $A$.


\subsection{Ultraproducts}

We will now define the \textit{regular ultraproduct} of $G$-probability spaces. Let us fix a (non principal) ultrafilter $\ul$. 

\begin{dfn}
  Let $\{X_n\}_{n\in\bn}$ be a sequence of sets and let $X$ be their Cartesian product $X\coloneqq \prod_{n\in\bn} X_n$. The \defin{ultraproduct} of the sequence $\{X_n\}_n$ is the set $[X_n]_\ul\coloneqq \prod_{n\in \bn}X_n/\sim_\ul$ where the equivalence relation $\sim_\ul$ is defined as usual by saying that $(x_n)_n \sim_\ul (y_n)_n$ if $x_n=y_n$ for $\ul$-almost every $n$.
\end{dfn}

When the sequence $\{X_n\}_n$ is clear from the context we put $X_\ul\coloneqq [X_n]_\ul$. We will denote by $x_\ul$ and $A_\ul$ elements and subsets of $X_\ul$. For a sequence $\{x_n\in X_n\}_n$ we will denote by $[x_n]_\ul$ its class in $X_\ul$ and similarly for a sequence of subsets $\{A_n\subseteq X_n\}_n$ we will denote its class by $[A_n]_\ul$. If each $X_n$ is a probability space, then there is a canonical way to construct a probability measure on $X_\ul$, the Loeb probability space, see for example \cite{car2015}. 

\begin{thm} 
Let $\{(X_n,\rb_n,\mu_n)\}_{n\in\bn}$ be a sequence of probability spaces and let $X_\ul$ be the ultraproduct of the sequence $\{X_n\}_n$. Then there exists a complete probability space $(X_\ul,\rb_\ul,\mu_\ul)$ such that
  \begin{enumerate}
  \item for every sequence $\{A_n\in \rb_n\}_n$ the set $[A_n]_\ul$ is in $\rb_\ul$ and $\mu_\ul([A_n]_\ul)=\lim_{\ul}\mu_n(A_n)$,
  \item for every $A_\ul\in\rb_\ul$ there exists a sequence $\{A_n\in \rb_n\}_n$ such that $\mu_\ul(A_\ul\Delta [A_n]_\ul)=0$.
  \end{enumerate}
\end{thm}

From now on we fix a \lcsc group $G$ and we assume that it acts on the sequence of probability spaces $\{(X_n,\rb_n,\mu_n)\}_{n\in\bn}$ preserving the measure. We can define an action of $G$ on the ultraproduct by the formula $g[x_n]_\ul\coloneqq [gx_n]_\ul$. This action preserves the measure $\mu_\ul$ on $X_\ul$. However even if the action of $G$ on $X_n$ is continuous for every $n$, the action of $G$ on the ultraproduct $X_\ul$ may be neither continuous nor measurable. The main point of this section is to define a continuous factor of the ultraproduct action.

\begin{dfn}\label{dfn:regset}
  Assume that the \lcsc group $G$ acts measurably on the sequence of probability spaces $\{(X_n,\mu_n)\}_n$. We set $[X_n]_\ul^R\coloneqq \prod_{n\in \bn}X_n/\sim^R_\ul$ for the equivalence relation defined by $(x_n)_n \sim^R_\ul (y_n)_n$ if there is a sequence $\{g_n\}_n$ of elements of $G$ such that $\lim_\ul g_n=\id_G$ and $g_nx_{n}= y_{n}$ for $\ul$-almost every $n$.
\end{dfn}

If the sequence $\{X_n\}_n$ of spaces is clear from the context we will put $X_\ul^R\coloneqq [X_n]_\ul^R$. We will denote by $[x_n]_\ul^R\in X_\ul^R$ the point associated with the sequence $(x_n)_n$. We will often denote by $A_\ul^R$ subsets of $X_\ul^R$. Put $G^0_\ul\coloneqq \left\{(g_n)_n\in \prod_{n\in\bn}G\colon \lim_\ul g_n=\id_G\right\}.$ Since $G$ is a topological group multiplication and inversion are continuous in $G$ and therefore $G^0_\ul$ is a group. Whenever $G$ acts on the sequence of spaces $\{(X_n,\mu_n)\}_n$  we can define an action of $G^0_\ul$ on $X_\ul$ by $(g_n)_n[x_n]_\ul\coloneqq [g_nx_n]_\ul.$ Observe that $[X_n]_\ul^R=[X_n]_\ul/G_\ul^0.$ We will denote by $\pi_\ul^R:[X_n]_\ul\rightarrow [X_n]_\ul^R$ the associated quotient map. For every $g\in G$ we have that $gG^0_\ul g^{-1}=G^0_\ul$ so the group $G$ acts on $[X_n]_\ul^R$ by the same formula $g[x_n]_\ul^R\coloneqq \pi_\ul^R(g[x_n]_\ul)=[gx_n]_\ul^R$ and $\pi_\ul^R$ is $G$-invariant. 

If $X_n=G/\Gamma_n$ for some sequence of lattices $(\Gamma_n)_n$ and $\mu_n$ is the normalized Haar measure, then $[X_n]_\ul^R$ corresponds to the metric ultraproduct of the sequence $(X_n)_n$ with respect to a suitable proper metric. Observe however that if $\covol(\Gamma_n)$ tends to infinity, then the limit metric on $[X_n]_\ul^R$ is only defined with values in $[0,+\infty]$. Two points in $[X_n]_\ul^R$ are at bounded distance for this limit metric if and only if they are in the same $G$-orbit. We will now define the measurable structure and the probability measure on $[X_n]_\ul^R$.

\begin{dfn}\label{dfn:regseq}
We say that a sequence of subsets $\{A_n\subseteq X_n\}_n$ is ($\ul$-)\defin{regular} if for every $\eps>0$ there exists a neighborhood of the identity $U\subseteq G$ such that $UA_n$ is measurable for $\ul$-almost every $n$ and for $\ul$-almost every $n$ we have $\mu_n(UA_n\setminus A_n)\leq \eps$.
\end{dfn}

Remark that for every sequence of subsets $\{A_n\subseteq X_n\}_n$ and neighborhood of the identity $U\subseteq G$ we have that $G_\ul^0[A_n]_\ul\subseteq [UA_n]_\ul$. Therefore if $\{A_n\}_n$ is regular we have that $G_\ul^0[A_n]_\ul\in \rb_\ul$ and $\mu_\ul(G_\ul^0[A_n]_\ul)=\mu_\ul([A_n]_\ul)$. For $\{A_n\}_n$ regular we put $[A_n]_\ul^R\coloneqq \pi_\ul^R([A_n]_\ul)$ and we will say that $[A_n]_\ul^R$ is a \defin{regular subset} of $[X_n]_\ul^R$.

\begin{dfn}
  Assume that the \lcsc group $G$ acts measurably on the sequence of probability spaces $\{(X_n,\mu_n)\}_n$ and let $[X_n]_\ul^R$ be as in Definition \ref{dfn:regset}. We let $\rb_\ul'$ be the smallest $\sigma$-algebra of subsets of $[X_n]_\ul^R$ which is generated by the sets $[A_n]_\ul^R$ for regular sequences $\{A_n\}_n$. Then we have that $(\pi_\ul^R)^{-1}(\rb_\ul^R)\subseteq \rb_\ul$ and we let $\mu_\ul^R\coloneqq (\pi_\ul^R)_*(\mu_\ul)$, that is $\mu_\ul^R(A^R_\ul)=\mu_\ul((\pi_\ul^R)^{-1}(A_\ul^R))$. We let now $\rb_\ul^R$ be the $\mu_\ul^R$-completion on the $\sigma$-algebra $\rb_\ul'$. We will call the measure space $([X_n]_\ul^R,\rb_\ul^R,\mu_\ul^R)$ the \defin{regular ultraproduct} of the sequence of probability spaces $\{(X_n,\rb_n,\mu_n)\}_n$.
\end{dfn}
In particular if $\{A_n\}_n$ is regular, we have that \[\mu_\ul^R([A_n]^R_\ul)=\mu_\ul(G_\ul^0[A_n]_\ul)=\mu_\ul([A_n]_\ul)=\lim_\ul\mu_n(A_n),\] and the action of $G$ on $([X_n]_\ul^R,\mu_\ul^R)$ is measure preserving.

Let $G$ be a \lcsc group and let $(H_n)_n$ be a sequence of subgroups of $G$. We say that $(H_n)_n$ converges to the subgroup $H$ of $G$ with respect to the \textit{Hausdorff pointed topology}, if for every compact subset $K\subseteq G$ and neighborhood of the identity $U\subseteq G$, for every $n$ big enough, we have $K\cap (UH_n)\supset K\cap H$ and $K\cap (UH)\supset K\cap H_n$. Remark that every sequence of subgroups $(H_n)_n$ admits a converging sub-sequence. Indeed given $(H_n)_n$ set $H_\ul\coloneqq\{\lim_\ul h_n\colon h_n\in H_n\}$. Then it is easy to remark that $H_\ul$ is a closed subgroup and $(H_n)_n$ converges to $H_\ul$ with respect to the Hausdorff pointed topology along the ultrafilter $\ul$. 

\begin{lem}\label{lem:stabil}
  Let $G$ be a \lcsc group. Suppose that $G$ acts in a Borel manner on the sequence of standard probability spaces $\{(X_n,\mu_n)\}_n$ preserving the measure. For every sequence $\{x_n\in X_n\}_n$ the stabilizer $G_{x_n}$ converges to $G_{[x]_\ul^{R}}$ in the Hausdorff pointed topology along the ultrafilter $\ul$.
\end{lem}
\begin{proof}
   Assume that for $\ul$-almost every $n$ we have that $g_nx_n=x_n$ and that the sequence $g_n$ is bounded. Put $g\coloneqq \lim_\ul g_n$ and observe that $(h_n)_n\coloneqq (gg_n^{-1})_n\in G^0_\ul$. So $g[x_n]_\ul^{R}=[gx_n]_\ul^{R}=[h_nx_n]_\ul^{R}=[x_n]_\ul^{R}.$ On the other hand assume that $g[x_n]_\ul^{R}=[x_n]_\ul^{R}$. Then there exists a sequence $(g_n)_n\in G^0_\ul$ such that $gx_n=g_nx_n$ for $\ul$-almost every $n$ whence $g_n^{-1}g\in G_{x_n}$ and $\lim_\ul g_n^{-1}g=g$. 
\end{proof}

The great advantage of the regular ultraproduct is that the group $G$ acts continuously and measurably on it, see Theorem \ref{thm:ultrabasic}. It follows quite easily from the definition that the action is continuous, however the proof that the action is measurable is more complicated. For this, we will have to understand regular subsets better. These sets behave like compact subsets in a standard probability space.

\begin{lem}\label{lem:apprfromabove}
   Let $\{A_n\}_n$ be a regular sequence of subsets and let $U_k$ be neighborhoods of the identity such that $\cap_k  U_k=\{\id_G\}$. Then $\cap_k   U_k[A_n]_\ul^R=[A_n]_\ul^R$. 
\end{lem}
\begin{proof}
 For the proof we will assume that $U_k$ is symmetric and that $U_k\supset U_{k+1}$ for every $k$. An element $[x_n]_\ul^R\in \cap_k U_k[A_n]_\ul^R$ if and only if for every $k$ there exists $u_k\in U_k$ and a sequence of elements $g^k_n$ such that $u_kg_n^kx_n\in A_n$ for $\ul$-almost every $n$ and $\lim_\ul g_n^k=\id_G$. Set $\alpha_k\coloneqq \{n\geq k\colon g_n^k\in U_k\text{ and }u_kg_n^kx_n\in A_n\}$. By assumption $\alpha_k\in \ul$ for every $k$. Set $k_n\coloneqq 1$ for all $n\notin \cup \alpha_k$ and define $k_n$ to be the maximum of all $k\in\bn$ such that $n\in\alpha_k$ otherwise. Put $g_n\coloneqq u_{k_n}g^{k_n}_n$. Observe that for every $n\in\alpha_1$ we have that $g_nx_n\in A_n$. Moreover for every $n\in\alpha_k$ we have that $g_n\in U_k^2$ and therefore $\lim_\ul g_n=\id_G$ which implies that $[x_n]_\ul^R\in [A_n]_\ul^R$ as claimed.
\end{proof}

\begin{lem}\label{lem:regularinter}
  \begin{itemize}
  \item Let $\{A_n\}_n$ and $\{B_n\}_n$ be regular sequences. Then $\{A_n\cap B_n\}_n$ and $\{A_n\cup B_n\}_n$ are regular sequences and $[A_n\cup B_n]_\ul^R=[A_n]_\ul^R\cup[B_n]_\ul^R$ and $[A_n\cap B_n]_\ul^R=[A_n]_\ul^R\cap[B_n]_\ul^R$.
  \item For every $i\in\bn$ we let $\{A_n^i\}_n$ be a regular sequence. Then there exists a regular sequence $\{A_n\}_n$ such that $[A_n]_\ul^R\subseteq \cap_i [A_n^i]_\ul^R$ and $\mu_\ul^R(\cap_i [A_n^i]_\ul^R\setminus [A_n]_\ul^R)=0$. 
  \end{itemize}
\end{lem}
\begin{proof}
  For every $U\subseteq G$ we have that both $U(A\cap B)\setminus (A\cap B)$ and $U(A\cup B)\setminus (A\cup B)$ are contained in $(UA\setminus A)\cup(UB\setminus B)$. Therefore if $\{A_n\}_n$ and $\{B_n\}_n$ are regular sequences, then $\{A_n\cap B_n\}_n$ and $\{A_n\cup B_n\}_n$ are regular sequences. If $[x_n]_\ul^R\in [A_n\cup B_n]_\ul^R$, then there exists an element $(g_n)_n\in G_\ul^0$ such that $g_nx_n\in A_n\cup B_n$ for $\ul$-almost every $n$. Hence for $\ul$-almost every $n$ the element $g_nx_n$ is either in $A_n$ or in $B_n$. The ultrafilter will choose one of the two and therefore $[x_n]_\ul^R\in [A_n]_\ul^R\cup [B_n]_\ul^R$. The argument for the intersection is similar.

Fix now for every $i\in\bn$ a regular sequence $\{A_n^i\}_n$. Observe that by the previous point for every $N\in\bn$ we have that $\cap_i^N[A_n^i]_\ul^R=[\cap_i^N A_n^i]_\ul^R$ and $\{\cap_i^N A_n^i\}_n$ is a regular sequence. Therefore we can assume that for every $i$ and $n\in\bn$ we have that $A^i_n\supset A^{i+1}_n$. Then a standard diagonal argument as in Lemma \ref{lem:apprfromabove} tells us that there exists a monotone function $f\colon\bn\rightarrow \bn$ such that $[A_n^{f(n)}]_\ul\subseteq \cap_i [A_n^i]_\ul$ and $\mu_\ul(\cap_i [A_n^i]_\ul\setminus [A_n^{f(n)}]_\ul)=0$. Set $A_n\coloneqq A_n^{f(n)}$ and let us prove that $\{A_n\}_n$ is regular. Fix $\eps>0$. There exists $N$ such that $\mu_n(A^N_n\setminus A_n)\leq \eps$ for $\ul$-almost every $n$. There also exists $U\subseteq G$ such that $\mu_n(UA_n^N\setminus A_n^N)\leq \eps$ for $\ul$-almost all $n$. Therefore $\mu_n(UA_n\setminus A_n)\leq \mu_n(UA_n^N\setminus A_n^N)+\mu_n(A^N_n\setminus A_n)\leq 2\eps.$
\end{proof}

\begin{lem}\label{lem:innout}
  Let $U\subseteq G$ be a neighborhood of the identity. If $\{A_n\}_n$ and $\{UA_n\}_n$ are regular sequences, then $[UA_n]_\ul^R=\overline U[A_n]_\ul^R$. Similarly if $\{X_n\setminus A_n\}_n$ and $\{X_n\setminus UA_n\}_n$ are regular, then $[X_n\setminus UA_n]_\ul^R=[X_n]_\ul^R\setminus \overset{\circ}{U}[A_n]_\ul^R$. 
\end{lem}
\begin{proof}
 Observe that $[y_n]_\ul^R\in [UA_n]_\ul^R$ if and only if there exists a sequence of elements $g_n\in G$ such that $g\coloneqq \lim_\ul g_n\in \overline U$ and element $x_n\in A_n$ such that $y_n=g_nx_n$ for $\ul$-almost every $n$. So $[y_n]_\ul^R=[gx_n]_\ul^R=g[x_n]_\ul^R$ and therefore $[UA_n]_\ul^R=\overline U [A_n]^R_\ul$. The other case is analogous.
\end{proof}

We fix a compatible metric $d$ on the \lcsc group $G$. We will set $B_r\coloneqq \{g\in G\colon d(g,\id_G)<r\}$ and $\overline B_r\coloneqq \{g\in G\colon d(g,\id_G)\leq r\}$.

\begin{lem}\label{lem:regularseq}
Assume that the \lcsc group $G$ acts measurably on the sequence of probability spaces $\{(X_n,\mu_n)\}_n$ preserving the measure. Consider a sequence of measurable subsets $\{A_n\}_n$ such that $EA_n\subseteq X_n$ is measurable for every Borel subset $E\subseteq G$. Then for almost all $r\in\br$ we have that $\{B_rA_n\}_n$, $\{\overline B_rA_n\}_n$, $\{X\setminus B_rA_n\}_n$ and $\{X\setminus \overline B_rA_n\}_n$ are regular sequences.
\end{lem}
\begin{proof}
  Let us prove the lemma for $\{B_rA_n\}_n$, the proof for the other cases is similar. Set $\alpha_n(r)\coloneqq \mu_n(B_rA_n)$ and observe that it is an increasing function. Therefore the function $\alpha_\ul(r)\coloneqq \lim_\ul\mu_n(B_rA_n)$ is increasing and hence continuous almost everywhere. Let $r$ be a point of continuity of $\alpha_\ul(r)$. We claim that $\{B_rA_n\}_n$ is a regular sequence. Indeed let us fix $\eps>0$. By continuity there exists $s\geq r$ such that $|\alpha_\ul(s)-\alpha_\ul(r)|<\eps/3$. For $\ul$-almost all $n$ we have that $|\alpha_n(s)-\alpha_\ul(s)|$ and $|\alpha_n(r)-\alpha_\ul(r)|$ are both small than $\eps/3$. Hence we get that for $\ul$-almost all $n$ we have that $|\alpha_n(s)-\alpha_n(r)|<\eps$ that is $|\mu_n(B_{s-r}(B_rA_n))-\mu_n(B_rA_n)|<\eps$ which implies that $\{B_rA_n\}_n$ is regular.
\end{proof}

We are now able to prove the following useful characterization. 

\begin{prop}\label{prop:innerregular}
  For every $A_\ul^R\in \rb_\ul^R$ and every $\eps>0$ there is a regular subset $[B_n]_\ul^R$ such that $[B_n]_\ul^R\subseteq A_\ul^R$ and $\mu_\ul^R(A_\ul^R\setminus [B_n]_\ul^R)\leq \eps$.
\end{prop}
\begin{proof}
  Let us denote by $\mathcal A$ the set of subsets $A_\ul^R\subseteq X_\ul^R$ for which both $A_\ul^R$ and its complement satisfy the condition in the statement. Observe that since the definition is symmetric $A_\ul^R\in \mathcal A$ implies that $X_\ul^R\setminus A_\ul^R\in\mathcal A$. Lemma \ref{lem:regularinter} implies that countable intersections of elements in $\mathcal A$ are in $\mathcal A$. Therefore $\mathcal A$ is a $\sigma$-algebra. Finally Lemma \ref{lem:apprfromabove} and \ref{lem:regularseq} tell us that every regular subset is in $\mathcal A$ and hence the $\sigma$-algebra generated by the regular subsets is contained in $\mathcal A$.
\end{proof}

Using the above lemmas, we can finally show that the group $G$ acts continuously and measurably on the complete probability space $([X_n]_\ul^R,\rb_\ul^R,\mu_\ul^R)$. 

\begin{thm}\label{thm:ultrabasic}
  Let $G$ be a \lcsc group. Suppose that $G$ acts in a Borel manner on the sequence of standard probability spaces $\{(X_n,\mu_n)\}_n$ preserving the measure. Then $G$ acts continuously and measurably on the regular ultraproduct $([X_n]_\ul^R,\rb_\ul^R,\mu_\ul^R)$. 
\end{thm}
\begin{proof}
  The action is clearly continuous at every regular subset $[A_n]_\ul^R\subseteq X_\ul^R$ and these subsets generate the $\sigma$-algebra. Therefore the action is continuous (which follows also from Proposition \ref{prop:automacont} and the measurability of the action). To show the measurability let us denote by $\Phi\colon G\times [X_n]_\ul^R\rightarrow [X_n]_\ul^R$ the action map. Proceeding as for standard Borel spaces, as explained after Definition \ref{dfn:measurability}, by Proposition \ref{prop:innerregular} it is enough to show that for every regular subset $[A_n]_\ul^R$ such that $[X_n\setminus A_n]_\ul^R$ is also regular we have that $\Phi^{-1}([A_n]_\ul^R)$ is measurable (cf. also with Lemma \ref{lem:apprfromabove} and \ref{lem:regularseq}).

Let $k\in \bn$ and let $U\subseteq G$ be a neighborhood of the identity such that $\{UA_n\}_n$ and $\{X_n\setminus U(X_n\setminus A_n)\}_n$ are regular and $\mu_n(UA_n)\leq \mu(A_n)+2^{-k}$ and $\mu_n(U(X_n\setminus A_n))\leq 1-\mu(A_n)+2^{-k}$. Consider a countable subset $\{g_i\}_i\subseteq G$ such that $\cup_i Ug_i^{-1}=G$ and for each $i$ let $D_i\subseteq Ug_i^{-1}$ be a measurable subset such that $\{D_i\}_i$ forms a measurable partition of $G$. Define 
\[ S_k\coloneqq \cup_i D_i\times g_i[X_n\setminus U(X_n\setminus A_n)]_\ul^R \text{ and } T_k\coloneqq \cup_i D_i\times g_i[UA_n]_\ul^R.\]
Observe that $S_k$ and $T_k$ are measurable and $S_k\subseteq \Phi^{-1}([A_n]_\ul^R)\subseteq T_k$. Put $S\coloneqq \cup_kS_k$ and $T\coloneqq \cap_kT_k$. Then we still have that $S\subseteq \Phi^{-1}([A_n]_\ul^R)\subseteq T$ and Fubini tells us that $T\setminus S$ is a null set (with respect to the product measure). Therefore $\Phi^{-1}([A_n]_\ul^R)$ is measurable.
\end{proof}

Finally we remark that one can prove that if a measurable subset $A_\ul\in \rb_\ul$ is $G_\ul^0$-invariant, then there exists a regular sequence $\{A_n\}_n$ such that $\mu_\ul([A_n]_\ul\Delta A_\ul)=0$. Therefore the $\sigma$-algebra generated by the regular subsets (in $[X_n]_\ul$) is up to null-sets the same as the sub-$\sigma$-algebra of $\rb_\ul$ of $G_\ul^0$-invariant sets. This common measure algebra is also the same as the one which is sometimes used in the context of von Neumann algebras, the measure algebra generated by the \textit{equicontinuous} sequences of subsets of $X_n$. In our work we will not use these different characterizations and hence we will not prove the equivalence. Let us also remark that we do not know whether these constructions actually define the same $\sigma$-algebra and whether this $\sigma$-algebra is the maximal continuous $\sigma$-algebra, the one that can be constructed thanks to Lemma \ref{lem:techcont}.


\section{Cross sections of ultraproducts}

From now on $G$ will be a \textit{unimodular} \lcsc group and we will fix a Haar measure $\haar$ on $G$. Whenever $G$ acts in a Borel way on a standard probability space $(X,\mu)$ preserving the measure and $Z\subseteq X$ is a Borel subset we will denote by $\Phi\colon G\times Z\rightarrow X$ the map induced by the action, that is $\Phi(g,z)=gz$. Clearly $\Phi(gh,z)=g\Phi(h,z)$. 

\begin{dfn}\label{dfn:cross}
  Fix a neighborhood of the identity $U\subseteq G$. Consider a Borel action of $G$ on the standard probability space $(X,\mu)$. A $U$-\defin{cross section} is a Borel subset $Y\subseteq X$ such that the map $\Phi:U\times Y\rightarrow X$ is injective and such that $\mu(X\setminus GY)=0$. 
\end{dfn}

We say that $Y$ is a cross section if it is a $U$-cross section for some neighborhood of the identity $U$. 
 We recall the following theorem of Forrest \cite{for1974}, see also \cite[Theorem 4.2]{kye2015}. 

\begin{thm}\label{thm:crossforrest}
  Every free \pmp action of a \lcsc group $G$ on a standard Borel space admits a cross section. 
\end{thm}

In the following we will fix a not necessarily free\footnote{the existence of a cross section forces the stabilizers to be discrete.} action of $G$ on the standard probability space $(X,\mu)$ and a cross section $Y\subseteq X$. We consider the equivalence relation \[\rel_Y\coloneqq \{(y_1,y_2)\in Y\times Y\colon\text{ there exists }g\in G\text{ such that }gy_1=y_2\}.\] Since the map $U\times Y\rightarrow X$ is injective $\rel_Y$ is an equivalence relation with countable classes. We recall some properties of $\rel$, see \cite[Proposition 4.3]{kye2015}.

\begin{prop}\label{prop:propercross}
  There exists a probability measure $\nu$ on $Y$ and a real number $0<\covol(Y)<\infty$ such that $\Phi_*(\haar\bigr|_U\times\nu)=\covol(Y)\mu\bigr|_{\Phi(U\times Y)}$. Moreover
 \begin{enumerate}[(1)]
 \item the equivalence relation $\rel_Y$ is Borel and the probability measure $\nu$ is $\mathcal R_Y$-invariant,
 \item $(\mathcal R_Y,\nu)$ is ergodic if and only if the action of $G$ is ergodic,
 \item $(\mathcal R_Y,\nu)$ has infinite orbits almost everywhere if and only if $G$ is non-compact.
 \end{enumerate}
\end{prop}

Our aim is to understand cross sections of the regular ultraproduct. Observe however that Definition \ref{dfn:cross} requires two structures on the standard probability space, the measurable structure and the Borel structure. This cannot be easily generalized to our setting. Therefore, in order to define the notion of cross section for regular ultraproducts, we need to modify the definition. 

\begin{dfn}\label{dfn:extcross}
  Let $G$ be a unimodular \lcsc group, fix a left Haar measure $\haar$ and a neighborhood of the identity $U$. Suppose that $G$ acts measurably on the probability space $(X,\mu)$ preserving the measure. An \defin{external} $U$-\defin{cross section} of \textit{covolume} $\covol(Y)$ is a probability space $(Y,\nu)$ and a measurable $G$-equivariant map $\Psi\colon G\times Y\rightarrow X$ such that 
  \begin{itemize}
  \item the subset $\Psi(U\times Y)\subseteq X$ is measurable,
  \item the restriction $\Psi\bigr|_{U\times Y}$ is a bijective measure preserving isomorphism between $(U\times Y,\haar\times \nu)$ and $(\Psi(U\times Y),\covol(Y)\mu)$,
  \item the measurable subset $\Psi(G\times Y)$ has full measure in $X$. 
  \end{itemize}
\end{dfn} 

We will show in Proposition \ref{prop:externalborel} that an external cross section of a Borel and probability measure preserving action on a standard Borel space is up to a null set a Borel cross section in the sense of Definition \ref{dfn:cross}. Also remark that since $\Psi$ is bijective and $\Psi(G\times Y)$ has full measure in $X$, the existence of an external cross section implies that almost every stabilizer of the action of $G$ on $X$ is discrete. 

\begin{prop}\label{prop:Tsection}
  In the notation of Definition \ref{dfn:extcross}, for any measurable subset $T\subseteq G\times Y$ we have that $\Psi(T)$ is measurable and if $\Psi\bigr|_T$ is injective, then $\covol(Y)\mu(\Psi(T))=\haar\times \nu(T)$.
\end{prop}
\begin{proof}
 Let $T\subseteq G\times Y$ be a measurable subset and let $\{g_j\}_{j\in\bn}$ be a dense subset of $G$. Since $\{g_j^{-1}U\times Y\}_j$ covers $T$ we can find measurable subsets $T_j\subseteq g_j^{-1}U\times Y$ in such a way $\{T_j\}_j$ is a partition of a conull subset of $T\setminus U\times Y$. Hence $\Psi(T_j)=g_j^{-1}\Psi(g_jT_j)$ is measurable and \[\mu(\Psi(T_j))=\mu(\Psi(g_jT_j))=\covol(Y)^{-1}\haar\times\nu(g_jT_j)=\covol(Y)^{-1}\haar\times\nu(T_j).\qedhere\]
\end{proof}

\begin{prop}\label{prop:Tgood}
    In the notation of Definition \ref{dfn:extcross}, there exists a measurable subset $T\subseteq G\times Y$ such that $T\supset U\times Y$ and such that the map $\Psi:(T,\haar\times\nu\bigr|_T)\rightarrow (X,\covol(Y)\mu)$ is a measure preserving isomorphism.
\end{prop}
\begin{proof}
  Let $\{g_j\}_{j\geq 1}$ be a countable dense subset of $G$ and set $g_0\coloneqq \id_G$. Set $T_0\coloneqq U\times Y$. Inductively define $T_j\coloneqq T_{j-1}\cup \Psi\bigr|_{g_jU\times Y}^{-1}(\Psi(g_jU\times Y)\setminus\Psi(T_{j-1}))$ and observe that $T\coloneqq \cup T_j$ satisfy the required conditions. 
\end{proof}

Proceeding as in the case of standard Borel spaces, we can define an equivalence relation on any external cross section. Indeed, let us suppose that the \lcsc unimodular group $G$ acts on the probability space $(X,\mu)$ preserving the measure and let $Y$ be an external cross section as in Definition \ref{dfn:extcross}. Then we can define an equivalence relation $\rel_Y$ on $Y$ by saying that $(y,y')\in \rel_Y$ if there exists $g\in G$ such that $g\Psi(\id_G,y)=\Psi(\id_G,y')$. The equivalence relation $\rel_Y$ is called the \defin{cross equivalence relation}. 

We will now show that this equivalence relation preserves the measure $\nu$: let us prove that there exists a countable family of partial measure preserving isomorphisms $\{\ph_j\}_j$ such that for every $(y,y')\in \rel_Y$ there is $j$ such that $\ph_j(y)=y'$. Let us define the partial maps $\ph$ first. 

\begin{prop}\label{prop:eqrelcross}
  Let $G$ be a unimodular \lcsc group. Suppose that $G$ acts measurably on the probability space $(X,\mu)$ preserving the measure. Let $Y$ be an external $U$-cross section for some neighborhood of the identity $U\subseteq G$.
    Consider $g\in G$ and an open subset $V\subseteq U$ such that $gVg^{-1}V\cup g^{-1}V^{-1}g\subseteq U$. Define \[D_g\coloneqq \left\{y\in Y\colon \Psi(g,y)\in\Psi(V\times Y)\right\}\]
and consider the map $\ph\colon D_g\rightarrow Y$ defined by 
\[\ph(y)\text{ is the unique }y'\in Y\text{ such that }\Psi(g,y)\in \Psi(V\times \{y'\}).\] 
   Then $D_g$ is measurable and $\ph$ is a bijective measure preserving isomorphism between $D_g$ and $\ph(D_g)$. 
\end{prop}
\begin{proof}
Let us first show that $D_g\subseteq Y$ is measurable. For every $k\in\mathbb N$, take a neighborhood of the identity $V_k\subseteq G$ such that $\bar V_k\subseteq V$ and $\cup_k V_k=V$. Then for every $k$ \[\Psi^{-1}(\Psi(V_k\times Y))=\{(g,y)\in G\times Y\colon \Psi(g,y)\in \Psi(V_k\times Y)\}=\cup_{g\in G}\{g\}\times D_g^k\]
where $D_g^k\coloneqq \left\{y\in Y\colon \Psi(g,y)\in\Psi(V_k\times Y)\right\}$. Since $\Psi^{-1}(\Psi(V_k\times Y))$
is measurable, Fubini's theorem implies that $D_g^k$ is measurable for almost every $g$ and every $k$. Now if we fix $g\in G$, there is a sequence $(g_n)_n$ converging to $g$ such that $D_{g_n}^k$ is measurable for every $k$ and $n$. Then using the equivariance of $\Psi$ and the fact that $V$ is open, one can show that \[D_g=\cup_k\cap_N\cup_{n\geq N}D_{g_n}^k\]
so that $D_g$ is measurable. 

Observe that since $\Psi\bigr|_{U\times Y}$ is injective, for every $y\in D$ there exist unique $v_y\in V$ and $y'\in Y$ such that $\Psi(g,y)=\Psi(v_y,y')$. So $\ph$ is well-defined. 

We claim that $\ph$ is injective. Take $y_1,y_2\in Y$ and let $v_i=v_{y_i}$ be such that $\Psi(g,y_i)=\Psi(v_i,\ph(y_i))$ as before. If $\ph(y_1)=\ph(y_2)$, then \[v_1^{-1}\Psi(g,y_1)=\Psi(\id_G,\ph(y_1))=\Psi(\id_G,\ph(y_2))=v_2^{-1}\Psi(g,y_2)\]
and hence $\Psi(g^{-1}v_1^{-1}g,y_1)=\Psi(g^{-1}v_2^{-1}g,y_2)$. By hypothesis $g^{-1}V^{-1}g\subseteq U$ and $\Psi\bigr|_{U\times Y}$ is injective, so $y_1=y_2$ as claimed.

Moreover since \[\Psi(gv,y)=gvg^{-1}\Psi(g,y)=gvg^{-1}\Psi(v_y,\ph(y))\]
the hypothesis that $gVg^{-1}V\subseteq U$ implies that 
\[g\Psi(V\times\{y\})=\Psi(gVg^{-1}v_y\times\{\ph(y)\})\subseteq \Psi(U\times Y).\]
That is, we showed that $g\colon\Psi(V\times D_g)\rightarrow g\Psi(V\times D_g)$ is a measurable isomorphism which maps ``horizontal lines to horizontal lines''. Therefore $\ph$ is a measurable isomorphism. Finally since $G$ is unimodular $\haar(gVg^{-1}v_y)=\haar(V)$. Moreover since $G$ preserves the measure an easy application of Fubini yields that $\ph$ is measure preserving.
\end{proof}

\begin{crl}\label{crl:cross is measurable}
There is a countable family of partial measure preserving isomorphisms $\{\ph_j\}_j$ such that for every $(y,y')\in \rel_Y$ there is $j$ such that $\ph_j(y)=y'$ 
\end{crl}
\begin{proof}
  Since $G$ is paracompact, there is a countable family $\{g_j\}_j\subseteq G$ and open subsets $V_j\subseteq U$ such that $gV_jg^{-1}V_j\cup g^{-1}V^{-1}_jg\subseteq U$ such that $\cup V_j^{-1}g_j=G$. For every $j$, define $\ph_j$ as in Proposition \ref{prop:eqrelcross} with respect to $g_j$ and $V_j$. 
  
  If $(y,y')\in\rel_Y$, then there is $g\in G$ such that $g\Psi(\id_G,y)=\Psi(\id_G,y')$. Consider $j$ such that $g=v^{-1}g_j$ with $v\in V_j$. Finally observe that $\ph_j(y)=y'$, since \[\Psi(g_j,y)=\Psi(vg,y)=\Psi(v,y')\in \Psi(V_j\times \{y'\}).\qedhere\]
\end{proof}

  Let us warn the reader that we have not proven that $\rel_Y$ is a measurable equivalence relation in the sense of \cite[Section 2.1]{CGS}. However, as a consequence of Theorem \ref{thm:mainrel}, we will get that the cross equivalence relation of the external cross sections of the ultraproduct is. Therefore we will be able to use the results in \cite{CGS}.

\begin{prop}
   In the notation of Definition \ref{dfn:extcross}, assume that the action of $G$ is free. Then the cross equivalence relation $\rel_Y$ on $Y$ constructed in the Proposition \ref{prop:eqrelcross} satisfies: 
   \begin{itemize}
   \item $(\mathcal R_Y,\nu)$ is ergodic if and only if the action of $G$ is ergodic;
   \item $(\mathcal R_Y,\nu)$ has infinite orbits almost everywhere if and only if $G$ is non-compact.
   \end{itemize}
\end{prop}

The proof is similar to the proof of Proposition \ref{prop:propercross} and therefore it is omitted. Let us finally show that an external cross section of a standard Borel space naturally induces a cross section in the space. For this we will need the following construction. Suppose now that the \lcsc group $G$ acts on the standard probability space $(X,\mu)$ in a Borel manner and let $A\subseteq X$ be a Borel subset. For every open neighborhood of the identity $U\subseteq G$ we define $\binte^U(A)$ to be the set of the $x\in X$ such that for almost every $g\in U$ we have $gx\in A$, that is \[\binte^U(A)\coloneqq\left\lbrace x\in X\colon \lambda(\{g\in U\colon gx\in A\})=\lambda(U)\right\rbrace.\]
The main point is that $\binte^U(A)$ is Borel. Indeed since the action is Borel, the set $\{(g,x)\in U\times X\colon gx\in A\}$ is Borel. Therefore Lemma 417B (b) of \cite{fremlin} tells us that the map $x\mapsto \haar(\{g\in U\colon gx\in A\})$ is Borel which directly implies that $\binte^U(A)$ is Borel.

\begin{prop}\label{prop:externalborel}
   Let $G$ be a \lcsc unimodular group, fix a Haar measure $\haar$ and a neighborhood of the identity $U$. Consider a Borel action of $G$ on the standard probability space $(X,\mu)$ and let $(Y,\Psi)$ be an external $U$-cross section. Then there exists a full measure subset $Y_0\subseteq Y$ such that $Z\coloneqq \Psi(\{\id_G\}\times Y_0)\subseteq X$ is a Borel $U$-cross section. Moreover the equivalence relation on $Y$ constructed before Proposition \ref{prop:eqrelcross} is orbit equivalent via $\Psi$ to the one on $Z$ of Proposition \ref{prop:propercross}.
\end{prop}

The proof of the proposition when $G$ has no small subgroups is quite straightforward. Indeed take a Borel subset $A'\subseteq \Psi(U\times Y)$ of full measure and Fubini implies that $Z\coloneqq \binte^U(A')\subseteq \Psi(\{\id_G\}\times Y)$ is a Borel $U$-cross section. In general we need to work a bit more. 

\begin{proof}
 Let $W_i$ be compact neighborhoods of the identity such that $W_i\subseteq U$ and $\cap_i W_i=\{\id_G\}$. Consider the measurable subsets $A_i\coloneqq\Psi(W_i\times Y)\subseteq X$ and take a Borel subset $A_i'\subseteq A_i$ such that $\mu(A_i\setminus A_i')=0$. By the above discussions, the sets $Z_i\coloneqq\binte^{W_i}(A_i')$ are Borel. By Fubini, $\Psi^{-1}(Z_i)\cap \{\id_G\}\times Y$ has full measure. 
 Set $Z\coloneqq \cap_i Z_i$ and observe that \[\Psi^{-1}(Z)\cap (U\times Y)\subseteq \cap_i W_i\times Y\subseteq \{\id_G\}\times Y.\] The moreover part follows easily from the definitions.
\end{proof}

\subsection{Cross sections and standard factors}

Working outside the context of standard probability spaces has the big disadvantage that one cannot use many of the standard facts about ergodic theory. Many of the arguments can be easily adapted to the more general setting, as we have seen for cross sections. In some other cases, it is possible to reduce the problem to the context of standard probability spaces by the means of a \textit{standard factor}. This phenomenon is widely used in the context of model theory: the \textit{downward Löwenheim–Skolem theorem} roughly states that every model of a countable theory has a countable model. We will now state and prove a version of this theorem for actions of \lcsc groups, see Proposition \ref{prop:realfact}, and for actions with a fixed cross section, see Proposition \ref{prop:standardcross}. These results are not necessary to obtain our main theorems and applications, however these propositions will allow us to ignore some technicalities and, in some sense, always pretend we are working with standard Borel spaces.  

The following proposition is a slightly more general version of \cite[Proposition B.5]{zim1984}. We will not present the proof of it since it is basically the same as the case of standard Borel spaces. 

\begin{prop}\label{prop:zb5}
  Let $G$ be a \lcsc group. Suppose that $G$ acts on the probability space $(X,\rb_X,\mu)$ measurably and preserving the measure. Assume furthermore that $G$ acts on the standard Borel space $(Z,\eta)$ preserving the measure and in a Borel manner.  Let $T\colon X\rightarrow Z$ be an almost everywhere surjective measurable and measure preserving map such that for every $g\in G$ and almost every $x\in X$ we have that $T(gx)=gT(x)$. Then there exist a co-null subset $Z_0\subseteq Z$ and a measure preserving surjective map $T'\colon X\rightarrow Z_0$ which is $G$-equivariant and is almost everywhere equal to $T$. 
\end{prop}

Combining the above proposition with a well known theorem of Mackey, \cite{mac1962}, we get the following.

\begin{prop}\label{prop:realfact}
 Let $G$ be a \lcsc group which acts measurably on the probability space $(X,\rb_X,\mu_X)$ preserving the measure and let $\rb_X'\subset\rb_X$ be a $G$-invariant separable sub-$\sigma$-algebra. Then there exist a Borel action of $G$ on a standard probability space $(Z,\rb_Z,\mu_Z)$ and a measure preserving, $G$-equivariant map $T\colon X\rightarrow Z$ which induces an isomorphism of measure algebras between $(X,\rb_X',\mu_X)$ and $(Z,\rb_Z,\mu_Z)$. 
\end{prop}
\begin{proof}
  Let $\{B_i\}_{i\in\bn}$ be a countable generating set for $\rb'_X$. Define $T\colon X\rightarrow \{0,1\}^\bn$ by $T(x)\coloneqq(\chi_{B_i}(x))_i$. Then the preimage of every Borel set of $Z\coloneqq\{0,1\}^\bn$ is in $\rb_X'$. We define a measure $\mu_Z$ on $Z$ by pushing down the measure $\mu_X$. Mackey's theorem, \cite{mac1962}, implies that there exists a Borel, measure preserving action of $G$ on $(Z,\mu_Z)$ which represents the action of $G$ on (the measure algebra of) $(X,\rb'_X,\mu_X)$. The measurable map $T$ is not necessarily $G$-invariant but if $g\in G$, then $(Tg)^{-1}$ and $(gT)^{-1}$ represent the same function at the level of measure algebras so that $T$ satisfies the hypothesis of Proposition \ref{prop:zb5}. 
\end{proof}

\begin{prop}\label{prop:standardcross}
   Let $G$ be a \lcsc unimodular group which acts measurably on the probability space $(X,\rb_X,\mu)$ preserving the measure and let $(Y,\rb_Y,\nu_Y,\Psi)$ be an external $U$-cross section, for some neighborhood of the identity $U\subseteq G$. Let $\rb$ be a separable sub-$\sigma$-algebra of $\rb_Y$. Then, up to removing a null-set, there are
\begin{itemize}
	\item a Borel action of $G$ on a standard probability space $(Z,\rb_Z,\mu_Z)$ and a measure preserving $G$-equivariant map $T\colon (X,\rb_X,\mu_X)\rightarrow (Z,\rb_Z,\mu_Z)$;
	\item a Borel cross section $(W,\rb_W,\nu_W)$ for the action of $G$ on $(Z,\rb_Z,\mu_Z)$ and a measure preserving map $S\colon (Y,\rb_Y,\nu_Y)\rightarrow (W,\rb_W,\nu_W)$ which is a factor of cross equivalence relations and which induces an injective map on $\rb\subseteq \rb_Y$;
\end{itemize}
   such that the diagram commutes
   \[\begin{tikzcd}
G\times Y \arrow{r}{\Psi} \arrow[swap]{d}{\mathrm{id}\times S} & X \arrow{d}{T} \\
G\times W \arrow{r}{\Phi} & Z
\end{tikzcd}
\]
   where $\mathrm{id}\colon G\rightarrow G$ is the identity map and $\Phi(g,y)=gy$ is the action map as defined before Definition \ref{dfn:cross}.
\end{prop}
\begin{proof}
  Let us fix a countable family $\{\ph_j\}_j$ as in Corollary \ref{crl:cross is measurable} for $\rel_Y$. Denote by $\rb'$ the separable $\sigma$-algebra generated by $\rb$, containing the domains of definition $\{D_j\}_j$ of $\{\ph_j\}_j$ and invariant by each $\ph_j$. Let $\mathcal D\subseteq \rb'$ be a countable generating subset invariant by the family $\{\ph_j\}_j$. As in the proof of Proposition \ref{prop:realfact}, we set $W\coloneqq \{0,1\}^{\mathcal D}$ and let $S\colon Y\rightarrow W$ defined by $S(x)\coloneqq (\chi_D(x))_{D\in\mathcal D}$. Equip $W$ with the push-forward measure denoted $\nu_W$ and denote the $\sigma$-algebra of measurable subsets by $\rb_W$. Clearly the measure algebras of $\rb_W$ and $\rb'$ are the same (up to null-sets) and hence for every $j$ we have a partially defined measure preserving isomorphism $\hat\ph_j$ on $W$. Denote by $\rel_W$ the equivalence relation generated by the family $\{\hat\ph_j\}_j$ and clearly $S$ maps $\rel_Y$-classes to $\rel_W$-classes. The equivalence relation $\rel_W$ is the standard factor as constructed in \cite[Theorem 3.28]{CGS}.
  
  Denote by $\rb_G$ the $\sigma$-algebra of measurable subsets on $G$ with respect to the Haar measure. Then $\Psi(\rb_G\times \rb')$ is a $G$-invariant sub-$\sigma$-algebra of measurable subsets of $\rb_X$. Hence Proposition \ref{prop:realfact} gives us a standard probability space $(Z,\rb_Z,\mu_Z)$ on which $G$ acts in a Borel manner and preserving the measure and a measure preserving, $G$-equivariant map $T\colon X\rightarrow Z$. 
  
  The $G$-equivariant map $\Psi$ can be restricted to the $\sigma$-algebra $\rb_G\times \rb'$. Since $W$ and $Z$ are both standard, the map $\Psi$ induces a $G$-equivariant map $\Psi_W\colon G\times W\rightarrow Z$ and it is easy to check that $(W,\Psi_W)$ is an external cross section of the action of $G$ on $Z$. Clearly $\Phi_W\circ(\mathrm{id}\times S)=T\circ \Psi$.
  
  We claim that $\rel_W$ is the cross equivalence relation. Indeed let us fix $j$. Consider $g_j\in G$ and $V_j\subseteq G$ which are used to define $\ph_j$. Assume that $\hat\ph_j(w)=w'$ and let $y\in Y$ such that $S(y)=w$. Then $\ph_j(y)=y'$ for some $y'$ such that $S(y')=w'$, that is $\Psi(g,y)=\Psi(v,y')$. By the commutativity of the diagram, we get that $\Psi_W(g,w)=\Psi_W(v,w')$ and so the claim is proven. 
  
  Finally, we can use Proposition \ref{prop:externalborel} to replace the external cross section with a Borel cross section, up to a null set. So the proof is complete. 
\end{proof}

We showed in \cite[Theorem 3.28]{CGS} that measurable equivalence relations with countable classes have \textit{nice} standard factors, that is class-bijective factors which are weakly equivalent to the starting equivalence relation. Now we can combine the theorem with Proposition \ref{prop:standardcross} to obtain that action of locally compact groups have \textit{nice} standard factors. In particular, the \textit{cost} (Definition \ref{dfn:costlc}) can be detected by standard factors.

\subsection{Ultraproducts of cross sections}

Now that we have collected enough information about cross sections outside the context of standard probability spaces, we can finally state and prove the main technical theorem of this work: the ultraproduct of cross sections is a cross section of the ultraproduct. This will allow us to reduce problems of convergence of invariants for actions of locally compact groups to invariants of their cross sections.

\begin{thm}\label{thm:maincross}
  Let $G$ be a \lcsc unimodular group, fix an open neighborhood of the identity $U'\subseteq G$ and a Haar measure $\haar$. Suppose that $G$ acts on the sequence of standard probability spaces $(X_n,\mu_n)$. Assume that for every $n$ there exists a $U'$-cross section $Y_n\subseteq X_n$ such that $\{\covol(Y_n)^{-1}\}_n$ is bounded and denote by $\nu_n$ the measure on $Y_n$ as in Proposition \ref{prop:propercross}. Take a compact neighborhood of the identity $U\subseteq U'$. Let $Y_\ul$ be the ultraproduct of $\{(Y_n,\nu_n)\}_n$. Consider the function
\begin{align*}
  \Psi\colon G\times Y_\ul\rightarrow X_\ul^R\quad\text{  }\quad\Psi(g,[y_n]_\ul)\coloneqq [\Phi_n(g,y_n)]_\ul^{R}.
\end{align*}
Then $(Y_\ul,\Psi)$ is an external $U$-cross section of the action of $G$ on the $G$-invariant measurable subset $G[\Phi_n(U\times Y_n)]_\ul^R\subseteq  X_\ul^{R}$ with covolume $\covol(Y_\ul)=\lim_\ul\covol(Y_n)$.
\end{thm}

Let us observe that $G[\Phi_n(U\times Y_n)]_\ul^R\subseteq X_\ul^R$ is measurable. First since $U\subseteq U'$ is compact it is not hard to show that the sequence $\{\Phi_n(U\times Y_n)\}_n$ is regular. Now take a sequence of open, bounded subsets $U_i\subseteq G$ such that the sequence $\{U_i\Phi_n(U\times Y_n)\}_n$ is regular and $\cup_i U_i=G$. Then by Lemma \ref{lem:innout} we have $\cup_i[U_i\Phi_n(U\times Y_n)]_\ul^R=\cup_i\overline U_i[\Phi_n(U\times Y_n)]_\ul^R=G[\Phi_n(U\times Y_n)]_\ul^R$. Furthermore remark that $\mu_n(\Phi_n(U\times Y_n))=\haar(U)/\covol(Y_n)$ and therefore the assumptions implies that $G[\Phi_n(U\times Y_n)]_\ul^R$ has positive measure. However it is possible that $G[\Phi_n(U\times Y_n)]_\ul^R$ has measure strictly less than $1$ and hence $(Y_\ul,\Psi)$ is not always an external cross section of $X_\ul^R$. 

For every $n\in\bn$ the Borel subset $Y_n\subseteq X_n$ is a $U$-cross section and we let $\rel_n$ be the cross equivalence relation on $Y_n$ constructed before Proposition \ref{prop:eqrelcross}. Theorem \ref{thm:maincross} tells us that the ultraproduct $Y_\ul$ is an external cross section of $X_\ul^{R}$ and we will denote by $\rel_\ul$ the cross equivalence relation on $Y_\ul$.  We will now prove that $\rel_\ul$ is the ultraproduct of the sequence $\rel_n$ with respect to a preferred sequence of graphings. Before stating the theorem, we will need to fix some terminology.

A \defin{graphing} $\Theta=\{\ph_j\}_{j\in J}$ of a \pmp equivalence relation $\rel$ on the probability space $(Y,\nu)$ is a countable collection of partial measure preserving isomorphisms $\{\ph_j\}_j$ of $Y$ such that 
\begin{itemize}
	\item for almost every $y\in Y$ and $j\in J$, if $\ph_j(y)$ is defined, then $(\ph_j(y),y)\in \rel$.
\end{itemize}
A graphing is said to be \defin{generating} if 	
\begin{itemize}
	\item the equivalence relation $\rel$ is the smallest equivalence relation containing the graph of each $\ph_j$ for $j\in J$.
\end{itemize} 

Clearly given a graphing $\Theta$ of an equivalence relation $\rel$, we can consider the sub-equivalence relation $\rel'$ \textit{generated} by $\Theta$, that is the sub-equivalence relation $\rel'$ for which $\Theta$ is a generating graphing. 

For cross sections we have a somewhat canonical graphing given by Corollary \ref{crl:cross is measurable}. 

\begin{dfn}\label{dfn:kgraph}
Consider a measurable and measure preserving action of the \lcsc group $G$ on the probability space $(X,\mu)$, let $(Y,\Psi)$ be an external $U$-cross section for some $U\subseteq G$. Let $\rel_Y$ be the cross equivalence relation on $Y$. Denote by $J$ a countable set. For every $j\in J$, take $g_j\in G$ and $V_j\subseteq G$ satisfying $g_jV_jg_j^{-1}V_j\cup g^{-1}_jV_j^{-1}g_j\subseteq U$. For every $j\in J$, Proposition \ref{prop:eqrelcross} grants us a partial measure preserving isomorphism $\ph_j$ and denote by $\Theta\coloneqq\{\ph_j\}_j$ the obtained graphing. Set $K\coloneqq\cup_jV_j^{-1}g_j$ and we will say that $\Theta$ is $K$-\defin{supported}. We say that the cross equivalence relation $\rel_Y$ is $K$-\defin{generated} if it admits a $K$-supported graphing. 

Similarly we will say that $\Theta$ is \defin{compactly supported} if it is $K$-supported for some pre-compact $K$ and $\rel_Y$ is \defin{compactly generated} if it admits a compactly supported generating graphing. 
\end{dfn}

A remark is in order. Assume that we have $g,V$ and $\ph$ as in Proposition \ref{prop:eqrelcross}. Then for every $y$ in the domain of definition of $\ph$, we have that $\Psi(\id_G,\ph(y))=\Psi(v^{-1}g,y)$ for some $v\in V$. Whenever $G$ acts freely on a standard probability space, then we obtain that $\Theta$ is $K$-supported if and only if 
\begin{itemize}
    \item if $\ph\in \Theta$, the for every $y$ in the support of $\ph$, we have that $\ph(y)=gy$ for some $g\in K$;
    \item whenever $g\in K$ is such that $gy\in Y$, there is $\ph\in\Theta$ such that $\ph(y)=gy$.
\end{itemize}

We also say that the cross section $Y$ is \defin{cocompact} if there exists a compact subset $K'\subseteq G$ such that $\mu(\Psi(K'\times Y))=1$. Cross equivalence relations of cocompact cross sections are always compactly generated, see Lemma \ref{lem:kpvgen}.

Let us now define the ultraproduct of equivalence relations, cf.\ Section 2.7 of \cite{CGS}. For every $n\in \mathbb N$, let $\rel_n$ be a \pmp\ equivalence relation on the probability space $(Y_n,\nu_n)$. Let $J$ be a countable set and for each $n$, let $\Theta_n=\{\ph_n^j\}_{j\in J}$ be a graphing of $\rel_n$. Then for each $j$, we can define a partial measure preserving isomorphism $\ph_\ul^j$ of the ultraproduct space $(Y_\ul,\nu_\ul)$ via $\ph_\ul^j\coloneqq [\ph_n^j]_\ul$, that is $\ph_\ul^j([y_n]_\ul)=[\ph_n^j(y_n)]_\ul$ whenever $y_n$ is in the domain of $\ph_n^j$. Put $\Theta_\ul\coloneqq \{\ph_\ul^j\}_j$ and let $\rel_\ul$ be the \pmp\ equivalence relation on $Y_\ul$ generated by it. The couple $(\rel_\ul,\Theta_\ul)$ is the \defin{ultraproduct equivalence relation} of the sequence of graphed equivalence relations $(\rel_n,\Theta_n)$. 

\begin{thm}\label{thm:mainrel}
   Let $G$ be a \lcsc unimodular group, fix an open neighborhood of the identity $U'\subseteq G$ and a Haar measure $\haar$. Suppose that $G$ acts on the standard probability spaces $(X_n,\mu_n)$ preserving the measure and assume that for every $n$ there exists a $U'$-cross section $Y_n\subseteq X_n$ such that $\{\covol(Y_n)^{-1}\}_n$ is bounded. Denote by $\rel_\ul$ and $\rel_n$ the cross equivalence relations on $Y_\ul$ and $Y_n$ respectively. Then there are generating graphing $\Theta_\ul=\{\ph^j_\ul\}_j$ and $\Theta_n=\{\ph^j_\ul\}_j$ of $\rel_\ul$ and $\rel_n$ respectively such that 
 \begin{itemize}
	\item the ultraproduct of the sequence of graphed equivalence relations $(\rel_n,\Theta_n)$ is $(\rel_\ul,\Theta_\ul)$.
\end{itemize}
   
   Assume moreover that $G$ is compactly generated and there is a compact subset $K'$ such that $\lim_\ul\mu_n(K'Y_n)=1$. Then there is a compact subset $K\subseteq G$ containing $K'$ and $K$-supported graphings $\Theta_\ul^K=\{\ph^j_\ul\}_j$ and $\Theta_n^K=\{\ph^j_\ul\}_j$ of $\rel_\ul$ and $\rel_n$ respectively such that 
\begin{itemize}
	\item the graphing $\Theta_\ul^K$ of $\rel_\ul$ is generating and in particular $\rel_\ul$ is compactly generated;
    \item the ultraproduct of the sequence of graphed equivalence relations $(\rel_n^K,\Theta_n^K)$ is $(\rel_\ul,\Theta_\ul^K)$.
   \end{itemize}
\end{thm}

We remark here that in the moreover part the graphings $\Theta^K_n$ of $\rel_n$ are not necessarily generating.

\subsection{Proof of Theorem \ref{thm:maincross} and Theorem \ref{thm:mainrel}}

We start by proving Theorem \ref{thm:maincross}. In the proof we will set \[\Psi_U\coloneqq \Psi\bigr|_{U\times Y_\ul}\text{ and }\Phi_{n,U}\coloneqq \Phi\bigr|_{U\times Y_n}.\]

Let us start by showing that $\Psi_U$ is injective. For this, take $u,u'\in U$ and $[y_n]_\ul,[y'_n]_\ul\in Y_\ul$. If \[[\Phi_n(u,y_n)]_\ul^{R}=\Psi(u,[y_n]_\ul)=\Psi(u',[y_n']_\ul)=[\Phi_n(u',y_n')]_\ul^{R}\] then there exists a sequence $(g_n)_n$ such that $\lim_\ul g_n=\id_G$ and $g_n\Phi(u,y_n)=\Phi_n(u',y'_n)$. For $\ul$-almost every $n$ we have that $g_nu\in U'$ and since $Y_n$ is a cross section we get that $g_nu=u'$ and $y_n=y_n'$ for $\ul$-almost every $n$ and therefore $u=u'$. Hence $\Psi_U$ is injective as claimed. 

We will now prove that $\Psi_U$ is a measurable isomorphism with its image. If $B\subseteq U'$ is a compact subset and $[C_n]_\ul\subseteq Y_\ul$ is a measurable subset, then $\{\Phi_n(B\times C_n)\}_n$ is regular. Indeed let us fix $\eps>0$. Then there exists a neighborhood of the identity $V\subseteq G$ such that $\haar(VB)\leq \haar(B)+\eps$ and such that $VB\subseteq U'$. Put $c\coloneqq \lim_\ul \covol(Y_n)^{-1}\in \br\setminus \{0\}$. For $\ul$-almost every $n$ we have that \[\mu_n(V\Phi_n(B\times C_n))=\covol(Y_n)^{-1}\haar(VB)\nu_n(C_n)\leq \covol(Y_n)^{-1}\haar(B)\nu_n(C_n)+(3c/2)\eps.\] In particular since $\{\Phi_n(\{\id_G\}\times C_n)\}_n$ and $\{\Phi_n(B\times C_n)\}_n$ are regular, Lemma \ref{lem:innout} yields \[\Psi(B\times [C_n]_\ul)=B\Psi(\{\id_G\}\times [C_n]_\ul)=B[\Phi_n(\{\id_G\}\times C_n)]_\ul^R=[\Phi_n(B\times C_n)]_\ul^R.\]

Since the subset of the form $B\times [C_n]_\ul$ generated the measure algebra of $U\times Y_\ul$ the above computation implies that for every measurable subset $A_\ul\subseteq U\times Y_\ul$ we have that $\Psi(A_\ul)$ is measurable and $\mu_\ul(\Psi(A_\ul))=c\haar\times \nu_\ul(A_\ul)$. Therefore $\covol(Y_\ul)=c^{-1}=\lim_\ul\covol(Y_n)$. In order to conclude the proof we have to show that the for every $D_\ul^R\subseteq \Psi(U\times Y_\ul)$ we have that $\Psi_U^{-1}(D_\ul^R)$ is measurable. 

As a first step let us assume that $D_\ul^R=[D_n]_\ul^R$ is regular. Consider $E_n\coloneqq\Phi_{n,U}^{-1}(D_n)$ and for every $y_n\in Y_n$ set \[E^{y_n}_n\coloneqq\{g\in U\colon (g,y_n)\in E_n\}.\]
 By Lemma \ref{lem:apprfromabove}, \ref{lem:innout} and \ref{lem:regularseq} we can assume that $E^{y_n}_n$ is compact for every $n$ and $y_n\in Y_n$. Since the Hausdorff topology on the compact subsets of $U$ is compact, for every $[y_n]_\ul\in Y_\ul$ the sequence $(E^{y_n}_n)_n$ converges along the ultrafilter $\ul$ to a compact subset $E^{[y_n]_\ul}_\ul\subseteq U$. Put $E_\ul\coloneqq\cup_{[y_n]_\ul\in Y_\ul}E^{[y_n]_\ul}_\ul$. We claim that $\Psi_U^{-1}(D_\ul^R)=E_\ul$. Indeed $\Psi(g,[y_n]_\ul)=[\Phi_n(g,y_n)]_\ul^R\in D_\ul^R$ if and only if there exists a sequence $g_n$ such that $\lim_\ul g_n=\id_G$ and for $\ul$-almost every $n$ we have that $(g_ng,y_n)\in E_n$. Therefore $g_ng\in E_n^{y_n}$ and the limit $\lim_\ul g_ng=g \in E^{[y_n]_\ul}_\ul$. Let us show now that $E_\ul$ is measurable. For this take for every $j\in\bn$ a compact neighborhood of the identity $W_j$ such that $\cap_j W_j=\{\id_G\}$. For $K,K'\subseteq U$, we set 
 \[N(K,K')\coloneqq \max\{j\in\mathbb N\colon W_jK\supset K'\text{ and }W_jK'\supset K\}\in\mathbb N\cup\{\infty\}\]
 and set $N(K,K')$ to be $0$ if no such $j$ exists. By compactness of the Hausdorff topology, for every $j$ there are compact subsets $K_j^1,\ldots,K_j^{m_j}$ of $U$ such that for every compact subset $K$ of $U$ we have that there exists $i$ satisfying $N(K,K^i_j)\geq j$. Set \[C_n^{j,1}\coloneqq\{y_n\in Y_n\colon N(E_n^{y_n},K^1_j)\geq j\}\] and recursively \[C_n^{j,i}\coloneqq\{y_n\in Y_n\colon N(E_n^{y_n},K^i_j)\geq j\}\setminus \cup_{h<i}C_n^{j,h}.\]
  Consider the measurable subset \[F_\ul^j\coloneqq\cup_{i=1}^{m_j}W_jK_j^i\times [C_n^{j,i}]_\ul\subseteq U\times Y_\ul.\]
   Clearly $F_\ul^j\supset E_\ul$. Observe also that $W_j^2E_\ul^{[y_n]_\ul}\supset W_jK_j^i$ whenever $[y_n]_\ul\in [C_n^{j,i}]_\ul$ and therefore $W_j^2 E_\ul\supset F^j_\ul$. Lemma \ref{lem:apprfromabove} implies that $\cap_jW_j^2D_\ul^R=D_\ul^R$ and hence \[ E_\ul=\Psi^{-1}(D_\ul^R)=\cap_j W_j^2 \Psi^{-1}(D_\ul^R)\supset \cap_j F^j_\ul\supset E_\ul.\]

Therefore for every regular subset $D_\ul^R$ we have that $\Psi_U^{-1}(D_\ul^R)$ is measurable. Let now $D_\ul^R\subseteq \Psi_U(U\times Y_\ul)$ be any measurable subset. By Proposition \ref{prop:innerregular} there are subsets $A_\ul^R$ and $B_\ul^R$ such that \[A_\ul^R\subseteq D_\ul^R,\quad B_\ul^R\cap D_\ul^R=\emptyset,\quad \mu_\ul(\Psi_U(U\times Y_\ul)\setminus (A_\ul^R\cup B_\ul^R))=0\] and such that both $A_\ul^R$ and $B_\ul^R$ are increasing union of regular subsets. Observe now that $\Psi_U^{-1}(A_\ul^R)$ and $\Psi_U^{-1}(B_\ul^R)$ are measurable and that $\mu_\ul(A_\ul^R)=c\haar \times \nu_\ul(\Psi^{-1}_U(A_\ul^R))$ and similarly for $B_\ul^R$. Hence $U\times Y_\ul\setminus (\Psi_U^{-1}(A_\ul^R)\cup\Psi_U^{-1}(B_\ul^R))$ is a null set and since the product measure is complete $\Psi_U$ is a measurable map and therefore a measurable, measure preserving isomorphism. So the proof of Theorem \ref{thm:maincross} is complete.

\vspace{0.3cm}

Before proceeding with the proof of Theorem \ref{thm:mainrel} we will prove the following lemma which is essentially Proposition 4.6 of \cite{kye2015}.

\begin{lem}\label{lem:kpvgen}
  Assume that the \lcsc group $G$ is compactly generated. Consider a \pmp action of $G$ on $(X,\mu)$ and let $(Y,\Psi)$ be an external cross section. If $\Psi(K\times Y)\subseteq X$ has full measure for some compact subset $K\subseteq G$, then there are compact subsets $K',K''$ such that $K''\supseteq K'\supseteq K$ and  
  \begin{itemize}
  \item $\Psi(K'\times Y)$ contains a full measure $G$-invariant subset,
  \item the cross equivalence relation $\rel_Y$ is $K''$-generated.
  \end{itemize}
\end{lem}
\begin{proof}
  Let $V\subseteq G$ be a symmetric compact neighborhood of the identity and set $K'\coloneqq VK$. Consider $A\coloneqq X\setminus \Psi(K'\times Y)$. Then $VA\subseteq X\setminus \Psi(K\times Y)$ and therefore $VA$ is a null set and hence $GA$ is. Therefore $\Psi(K'\times Y)\setminus GA$ is $G$-invariant and has full measure. 

 The second point can now be proven exactly as in Proposition 4.6 of \cite{kye2015}, therefore we only sketch it. Let $C\subseteq G$ be a compact generating subset such that $C=C^{-1}$ and set $K''\coloneqq (K')^{-1}CK'$. Consider $g\in G$ and $y\in Y'$ such that $\Psi(g,y)\in \Psi(\{\id_G\}\times Y)$. Since $C$ is generating, there are $g_1,\ldots,g_l\in C$ such that $g=g_l\ldots g_1$. Since $\Psi(g_i\ldots g_1,y)\in \Psi(K'\times Y)$ there are elements $k_i\in K'$ such that $\Psi(k_i^{-1}g_i\ldots g_1,y)\in \Psi(\{\id_G\}\times Y)$. Therefore we can set $h_0=k_1^{-1}g_1$ and $h_i=k_i^{-1}g_ik_{i-1}$ for $1<i<l$ and $h_l\coloneqq g_lk_{l-1}$. Then $h_i\in K''$ for every $i$ and $h_l\ldots h_1=g$. 
\end{proof}

Let us now prove Theorem \ref{thm:mainrel}. 
Fix a compatible right invariant metric $d_G$ on $G$. Consider a countable subset $\{g_j\}_{j\in J}\subseteq G$, a real number $\delta>0$ and open pre-compact neighborhoods of the identity $\{W_j\}_j$ such that \[g_jB_\delta W_jg_j^{-1}B_\delta W_j\cup g_j^{-1}(B_\delta W_j)^{-1}g_j\subseteq U\] as in Proposition \ref{prop:eqrelcross} and such that $\cup_j  W_j^{-1}g_j=G$. For $j\in J$ and $r\leq \delta$ we set 
\begin{align*}
  E^{j,r}_\ul\coloneqq&\left\{[y_n]_\ul\in Y_\ul\colon \Psi(g_j,[y_n]_\ul)\in\Psi(B_r W_j\times Y_\ul)\right\},\\
  D^{j,r}_n\coloneqq&\left\{y_n\in Y_n\colon \Phi_n(g_j,y_n)\in\Phi_n(B_r W_j\times Y_n)\right\}
\end{align*}
and observe that for every $0<r<r'<\delta$ we have \[ [D^{j,r}_n]_\ul\subseteq E^{j,r}_\ul \subseteq [D^{j,r'}_n]_\ul\subseteq E^{j,r'}_\ul.\]
We now proceed as in the proof of Lemma \ref{lem:regularseq}. The function $f\colon r\mapsto \nu_\ul(E^{j,r}_\ul)$ is monotone increasing and therefore there exists a $0< r_j<\delta$ such that $f$ is continuous at $r_j$. The above equation implies that $\mu^{R}_\ul([D^{j,r_j}_n]_\ul\Delta E^{j,r_j}_\ul)=0$.

Let us denote by $\psi^j$ and $\ph_n^j$ the partial isomorphism defined in Proposition \ref{prop:eqrelcross} for $Y_n$ and $Y_\ul$ respectively associated with $B_{r_j}W_j$. Let $\ph_\ul^j\coloneqq[\ph_n^j]_\ul$ the partial isomorphism defined on $[D_n^{j,r_j}]_\ul$ by the formula $\ph_\ul^j([y_n]_\ul)=[\ph_n^j(y_n)]_\ul$. We claim that $\ph_\ul^j=\psi^j$ up to a null set.  If $y_n\in D^{j,r_j}_n$, then \[\Psi(g_j,[y_n]_\ul)=[\Phi_n(g_j,y_n)]_\ul^{R}=[\Phi_n(v_n,\ph_n^j(y_n))]_\ul^{R}\] for some sequence $v_n\in B_{r_j}W_j$. Set $v\coloneqq\lim_\ul v_n\in \overline{B_{r_j}W_j}$ and observe that $[\Phi_n(v_n,\ph_n^j(y_n))]_\ul^{R}=[\Phi_n(v,\ph_n^j(y_n))]_\ul^{R}$. Therefore whenever $v\in B_{r_j}W_j$
\[\Psi(g_j,[y_n]_\ul)=\Psi(v,[\ph_n^j(y_n)]_\ul)=\Psi(v,\ph_\ul^j([y_n]_n)),\] and by injectivity we get that $\psi^j=\ph_\ul^j$ on $\cup_k [D^{j,r_j}_n]_\ul^{R}\cap E^{j,r_j-2^{-k}}_\ul $ which is conull in $E^{j,r_j}_\ul$ as claimed. Therefore $(\rel_\ul,\Theta_\ul)$ is the ultraproduct of the graphed equivalence relations $(\rel_n,\Theta_n)$.

Assume now that $G$ is compactly generated and that $\lim_\ul\mu_n(K'Y_n)=1$. Then Lemma \ref{lem:kpvgen} implies that the equivalence relation $\rel_\ul$ is $K''$-generated for some compact subset $K''\subseteq G$. Then, we can choose finitely many indices $j_0,\ldots,j_m\in J$ such that $K\coloneqq\cup_j B_{r_j}W_j\supseteq K''$. Now remark that $K$ is pre-compact, $\rel_\ul$ is $K$-generated and $(\rel_\ul,\Theta^K_\ul)$ is the ultraproduct of $\{(\rel_n^K,\Theta^K_n)\}_n$. Therefore the moreover part is proven and this concludes the proof of Theorem \ref{thm:mainrel}.

\section{Rank and cost}

The \textit{cost} of a measure preserving action of a countable group and more generally of a \pmp equivalence relation (with countable orbits) was defined by Levitt in \cite{lev1995} and widely studied in \cite{gab2000}, see also \cite{aas2017} or \cite{CGS} for a discussion in the context of general probability spaces. Let us recall the definition. Let $\rel$ be a measure preserving equivalence relation on the probability space $(X,\mu)$. As we have already recalled a graphing of $\rel$ is a countable set of measure preserving partially defined bijections of the probability space whose graphs are contained in $\rel$ and we define the \textit{cost of the graphing} to be the sum of the measures of the domains of the partially defined bijections. A graphing of $\rel$ is \textit{generating} if the smallest equivalence relation containing the graphing is (up to a null set) $\rel$ itself. The \textit{cost of the equivalence relation} $\rel$ is the infimum of the costs of all generating graphings of $\rel$.

In our work we will also need the notion of cost for measure preserving actions of unimodular locally compact groups. Observe however, as for the $\ell^2$-Betti numbers \cite{pet2013}, that this definition depends on a choice of a Haar measure on the group and it is not an invariant of orbit equivalence, see Theorem 1.20 of \cite{car2017}. 

\begin{dfn}\label{dfn:costlc}
  Let $G$ be a unimodular \lcsc group $G$ and let us fix a Haar measure $\haar$. Assume that $G$ acts on the standard probability space $(X,\mu)$ and fix a cross section $Y\subseteq X$ and denote by $\rel_Y$ the cross equivalence relation. We set \[\cost(G\curvearrowright X)\coloneqq 1+\frac{\cost(\rel_Y)-1}{\covol(Y)}.\] 
\end{dfn}

Observe that if $Y$ and $Y'$ are cross sections, then $\rel_Y$ and $\rel_{Y'}$ are strongly orbit equivalent with compression factor $\covol(Y)/\covol(Y')$, see \cite[Proposition 4.3]{kye2015}, and therefore the cost of a \pmp action of $G$ is well defined. As for countable groups we define the cost of the group to be the infimum of the cost of all of its free \pmp actions $\cost(G)\coloneqq \inf_{G\curvearrowright X}\cost(G\curvearrowright X)$ with respect to a fixed Haar measure. The \textit{fixed price problem} can also be posed in the context of \lcsc groups: does there exist a \lcsc unimodular group $G$ equipped with a fixed Haar measure $\haar$ and two free actions of $G$ which do not have the same cost? As an example amenable groups have fixed price $1$ (which does not depend on the measure) since every cross equivalence relation is amenable, \cite[Proposition 4.3]{kye2015}. It has also been announced that $\mathrm{SL}_2(\br)$ has fixed price. It is unknown whether $\mathrm{SL}_n(\br)$ has fixed price $1$ for $n\geq 3$ and even whether for any \lcsc group $G$ the group $\bz\times G$ has fixed price $1$.

The definition of cost clearly makes sense outside the context of standard probability space; one can use an external cross section instead of a cross section. Also remark that combining Proposition \ref{prop:standardcross} and Theorem 3.28 of \cite{CGS} we obtain that the cost of an action on a non standard probability space is the same as the cost of some of its standard factors. 

\begin{dfn}\label{dfn:thickness}
  Let $G$ be a \lcsc group. Suppose that $G$ acts on the probability measure space $(X,\mu)$ preserving the measure and let $U\subseteq G$ be a neighborhood of the identity. We set \[(X)_U\coloneqq\{x\in X\colon gx\neq x\ \forall g\in U\setminus\{\id_G\}\}.\]
   We say that the action is $U$-\defin{thick} if $(X)_U=X$ up to a null set.
\end{dfn}

Clearly if the action is free $(X)_U=X$ for every $U$ and so it is $U$-thick for every $U$. 

\begin{dfn}\label{dfn:farber sets}
  Let $G$ be a \lcsc group. We say that a sequence of actions on the probability spaces $\{(X_n,\mu_n)\}_n$ is $\ul$-\defin{Farber} if for every neighborhood of the identity $\lim_\ul\mu_n((X_n)_U)=1$ and the sequence is called \textit{Farber} if $\lim_n\mu_n((X_n)_U)=1$.
\end{dfn}

Clearly a sequence of actions is Farber if and only if it is $\ul$-Farber for every ultrafilter. Observe that Lemma \ref{lem:stabil} implies that if $\{(X_n,\mu_n)\}_n$ is $\ul$-Farber, then the action of $G$ on $[X_n]_\ul^R$ is essentially free.

\subsection{Lattices}

As always we will say that a discrete subgroup $\Gamma$ of a locally compact group $G$ is a \textit{lattice} if the quotient $G/\Gamma$ has finite Haar measure and we will denote this measure by $\covol(\Gamma)$. Let $G$ be a \lcsc group with Haar measure $\haar$ and let $\{\Gamma_n\}_n$ be a sequence of lattices of $G$. Observe that if $U\subseteq G$ is a neighborhood of the identity, then $(G/\Gamma)_U=\{g\in G/\Gamma\colon g\Gamma g^{-1}\cap U=\{\id_G\}\}$. Therefore we can restate Definition \ref{dfn:farber sets} in the following way. 

\begin{dfn}\label{dfn:Farberthinlatt}
We say that the sequence $\{\Gamma_n\}_n$ is ($\ul$-)\defin{Farber} or \defin{almost everywhere thick}   
for every neighborhood of the identity $U\subseteq G$ \[\lim_\ul\frac{\haar\left(\left\{g\in G/\Gamma_n\colon g\Gamma_n g^{-1}\cap U=\{\id\}\right\}\right)}{\covol(\Gamma_n)}=1.\]
Remark that if a sequence is $\ul$-Farber, then $\lim_\ul\covol(\Gamma_n)=\infty$.

Similarly we say that a lattice is $U$-\defin{thick} if $(G/\Gamma)_U=G/\Gamma$. Clearly every cocompact lattice is $U$-thick for some $U$ and non cocompact lattices are never $U$-thick for any $U$. We say that a sequence of lattices $\{\Gamma_n\}_n$ is \defin{nowhere thin} (sometimes called \textit{uniformly discrete}) if there exists a neighborhood of the identity $U$ such that for every $n\in\bn$ the lattice $\Gamma_n$ is $U$-thick.
\end{dfn}

Let $\Gamma<G$ be a lattice. A $U$-\defin{separated} subset $Y$ of $G/\Gamma$ is a finite subset such that for every $y\neq y'\in Y$ we have that $Uy\cap Uy'=\emptyset$. We will say that a $U$-separated subset $Y$ of $G/\Gamma$ is \defin{maximal} if it has maximal cardinality (which is always bounded by $\covol(\Gamma)/\haar(U)$).

\begin{lem}\label{lem:filling}
   Let $G$ be a unimodular \lcsc group and fix a Haar measure $\haar$. Let $\{\Gamma_n\}_n$ be a $\ul$-Farber sequence of lattices and let $U\subseteq G$ be a neighborhood of the identity. For every $n\in\bn$, let $Y_n\subseteq (G/\Gamma_n)_{U^{-1}U}$ be a maximal $U$-separated set. Then $Y_n$ is a $U$-cross section for the action of $G$ on $G/\Gamma_n$ of covolume $\covol(Y_n)=\covol(\Gamma_n)/|Y_n|$ such that 
   \[
    \haar(U^{-1}U)\geq \lim_\ul\covol(Y_n)\geq \haar(U)\text{ and }
   \lim_\ul\frac{\haar(U^{-1}UY_n)}{\covol(\Gamma_n)}=1.\]
\end{lem}
\begin{proof}
 Let $\Gamma$ be a lattice of $G$ and denote by $\mu$ the normalized Haar measure on $G/\Gamma$.
 If $y\in (G/\Gamma)_{U^{-1}U}$, then the map $u\in U \mapsto uy\in G/\Gamma$ is injective. Indeed if for $u_1,u_2\in U$, then $u_1y=u_2y$ if and only if $u_2^{-1}u_1y=y$ which is impossible if $y\in (G/\Gamma)_{U^{-1}U}$. Therefore $U$-separated subsets of $(G/\Gamma_n)_{U^{-1}U}$ are cross-sections. Also observe that if $Y\subseteq (G/\Gamma)_{U^{-1}U}$ is a $U$-separated subset, then $\mu(UY)=|Y|\haar(U)/\covol(\Gamma)$. That is, $\covol(Y)=\covol(\Gamma)/|Y|$ and 
  \[\frac{\haar(U)|Y|}{\covol(\Gamma)}=\mu(UY)\leq 1\Rightarrow \haar(U)\leq \covol(Y).\]
 
 If a $U$-separated subset $Y\subseteq (G/\Gamma)_{U^{-1}U}$ is maximal, then $U^{-1}UY\supseteq (G/\Gamma)_{U^{-1}U}$. Indeed given any $x\in (G/\Gamma)_{U^{-1}U}\setminus U^{-1}UY$, the subset $Y\cup \{x\}\subseteq (G/\Gamma)_{U^{-1}U}$ is still $U$-separated. In particular, \[\frac{|Y|\haar(U^{-1}U)}{\covol(\Gamma)}\geq \mu((G/\Gamma)_{U^{-1}U})\Rightarrow \haar(U^{-1}U)\geq \mu((G/\Gamma)_{U^{-1}U})\covol(Y).\]
 
 Now the lemma follows from the fact that if $\{\Gamma_n\}_n$ is a Farber sequence, we have that $\haar((G/\Gamma_n)_{U^{-1}U})/\covol(\Gamma_n)$ tends to $1$.
\end{proof}

Let $\{\Gamma_n\}_n$ be a $\ul$-Farber sequence of lattices of the unimodular, compactly generated \lcsc group $G$. Let us fix a neighborhood of the identity $U\subseteq G$. For each $n$, let $Y_n \subseteq (G/\Gamma_n)_{U^{-1}U}$ be a maximal $U$-separated set. Then we can apply Theorem \ref{thm:maincross} and Theorem \ref{thm:mainrel} to obtain that the cross equivalence relation on the ultraproduct is compactly generated and it is the ultraproduct of the finite equivalence relations on the cross sections $Y_n$. In this and in the next section we will show that this cross equivalence relation retains many properties of the sequence of lattices. 

\subsection{Cost and rank gradient}
For a countable group $\Gamma$ we will denote by $\rank(\Gamma)$ the minimal number of elements of $\Gamma$ which are needed to generate it.

 \begin{prop}\label{prop:upperbound}
   Let $G$ be a \lcsc group $G$ with fixed Haar measure $\haar$ and let $\Gamma<G$ be a lattice. Then $\cost(G)\leq 1+(\rank(\Gamma)-1)/\covol(\Gamma)$.
 \end{prop}
 \begin{proof}
   Consider a free \pmp action of $\Gamma$ on $(Y,\nu)$ and induce it to an action of $G$ on $X\coloneqq Y\times G/\Gamma$, see \cite[Definition 4.2.21]{zim1984}. Then $Y\subseteq X$ is a cross section of covolume $\covol(Y)=\covol(\Gamma)$ and therefore \[\cost(G\curvearrowright X)=1+\frac{\cost(\Gamma\curvearrowright Y)-1}{\covol(Y)}\leq 1+\frac{\rank(\Gamma)-1}{\covol(\Gamma)}\qedhere.\]
 \end{proof}

 The following theorem is an analogue for locally compact groups of \cite{abe2012}, \cite{tot2017} and \cite{CGS}.

\begin{thm}\label{thm:costgen}
  Let $G$ be a compactly generated unimodular \lcsc group and fix a Haar measure $\haar$. Let $\{\Gamma_n\}_n$ be a nowhere thin Farber sequence of lattices of $G$. Then \[\cost\left(G\curvearrowright [G/\Gamma_n]_\ul^R\right)\geq 1+\lim_\ul\frac{\rank(\Gamma_n)-1}{\covol(\Gamma_n)}\geq \cost(G).\]
\end{thm}

The second inequality follows from Proposition \ref{prop:upperbound} and hence if the group $G$ has fixed price the above inequalities are equalities.

 We will prove a stronger version of the theorem and for doing so we will need some notation. Let $G$ be a \lcsc group, let $\Gamma$ be a lattice and let $U\subseteq G$ be a neighborhood of the identity. Fix a $U$-cross section $Y$ of $G/\Gamma$ and assume that the projection of the identity is in $Y$. Let us now define the cross groupoid $\cg_Y$ associated with the cross section. The set of units of $\cg_Y$ consists of $Y$ and the elements of $\cg_Y$ are couples $(g,y)$ where $g\in G$, $y\in Y$ such that $gy\in Y$. Clearly it is a transitive groupoid and the stabilizers are conjugated to $\Gamma$. Fix a compact subset $U\subseteq K\subseteq G$. We define a graphing $\Theta^K$ exactly as in Definition \ref{dfn:kgraph}, $\Theta^K$ consists of the set of elements $(g,y)$ such that $g\in K$, $y\in Y$ and $gy\in Y$.
  The graphing is bounded\footnote{that is there exists a constant $D$ such that for every $y\in Y$ there are at most $D$ elements $g\in G$ such that $(g,y)$ is in the graphing.} by some constant which only depends on $U$ and $K$ but not the size of $Y$. We will denote by $\cg^K_Y$ the groupoid generated by $\Theta^K$, by $\rel^K_Y$ the associated equivalence relation and by $\Gamma^K<\Gamma$ the stabilizer of the projection of the identity. One can check that the cost of $\cg_Y^K$ as a groupoid is (see for example Lemma 21 of \cite{agn2017})\[\cost(\cg^K_Y)=1+\frac{\rank(\Gamma^K)-1}{|Y|}.\]

Denote by $\pi\colon G\rightarrow G/\Gamma$ the projection map and let $\widetilde Y\coloneqq \pi^{-1}(Y)$. The following lemma is straightforward. 

\begin{lem}\label{lem:notproven}
  An element $\gamma\in\Gamma$ is in $\Gamma^K$ if and only if there are elements $\tilde y_1,\ldots,\tilde y_h\in \widetilde Y$ such that $\tilde y_1=\id_G$, $\tilde y_h=\gamma$ and $\tilde y_{i+1}\in K\tilde y_i$ for every $i\leq h-1$. In particular if $K_0\widetilde Y=G$, then for every $K\supset K_0^{-1}K_0$ we have that $\Gamma^{K}=\Gamma$ and $\cg^{K}_Y=\cg_Y$. If $\Gamma$ is $U$-thick and $K\supset U^{-1}U$, then $\cg^K_Y=\cg_Y$.
\end{lem}

The following theorem is a strengthening of Theorem \ref{thm:costgen}.

\begin{thm}\label{thm:costgenp}
    Let $G$ be a compactly generated unimodular \lcsc group and fix a Haar measure $\haar$. Let $\{\Gamma_n\}_n$ be a Farber sequence of lattices of $G$ and let $U\subseteq G$ be a neighborhood of the identity. For every $n$ fix a maximal $U$-separated subset $Y_n\subseteq (G/\Gamma_n)_{U^{-1}U}$. Consider a compact subset $K'\subseteq G$ containing $U$. Then there exists $K\supseteq K'$ such that \[\cost\left(G\curvearrowright [G/\Gamma_n]_\ul^R\right)\geq 1+\lim_\ul\frac{\rank(\Gamma_n^K)-1}{\covol(\Gamma_n)}.\]
\end{thm}
\begin{proof}
  Denote by $\mu_n$ the renormalized Haar measure on $G/\Gamma_n$. By Lemma \ref{lem:filling}, $\{Y_n\}_n$ is a sequence of cross-sections. The group $G$ acts measurably on $[G/\Gamma_n]_\ul^R$ and by Theorem \ref{thm:maincross} we have that $Y_\ul$ is an external cross section. Moreover we can also choose $K\supset K'$ as in Theorem \ref{thm:mainrel} to obtain that the cross equivalence relation $\rel_\ul$ is the ultraproduct of the graphed equivalence relations $(\rel_n^K,\Theta^K_n)$. Since the sequence is Farber the action on $[G/\Gamma_n]_\ul$ is free and hence $\rel_\ul$ is also the ultraproduct of the sequence $(\cg_n^K,\Theta^K_n)$. Remark that $\Theta_n^K$ is uniformly bounded on $n$ and hence we can apply Theorem 3.13 of \cite{CGS} to obtain that $\cost(\rel_\ul)\geq 1+\lim_\ul\frac{\rank(\Gamma_n^K)-1}{|Y_n|}$. Observe that $\covol(Y_n)=\covol(\Gamma_n)/|Y_n|$ whence we get \[\cost\left(G\curvearrowright [G/\Gamma_n]_\ul^R\right)=1+\lim_\ul \frac{\rank(\Gamma_n^K)-1}{|Y_n|}\lim_\ul \frac{|Y_n|}{\covol(\Gamma_n)}=1+\lim_\ul\frac{\rank(\Gamma_n^K)-1}{\covol(\Gamma_n)}.\qedhere\]
\end{proof}

We would like now to identify the subgroup $\Gamma^K<\Gamma$ in the case of quasi connected groups. Remember that a \lcsc group is \textit{quasi connected} if one of its quotients by a compact subgroup is connected. In this case we also have that there exists a unique up to conjugacy maximal compact subgroup $K_0<G$ and that $K_0\backslash G$ is contractible, see \cite[Theorem 8]{glu1960} and \cite[Theorem A.5]{abe1974}. Put $M\coloneqq K_0\backslash G$. We will denote by $\rho$ the quotient map from $G$ to $M$ and by $\rho^\Gamma$ the quotient map from $G/\Gamma$ to $M /\Gamma$ where $\Gamma<G$ is any lattice.

 If $\Gamma$ is a torsion free lattice of $G$, then the fundamental group of $M/\Gamma$ is $\Gamma$ itself. The isomorphism is given as follows. Any loop (whose start point is the class of the identity) in $M/\Gamma$ can be lifted to a path in $M$ which starts at the class of the identity and ends at the class of an element of $\Gamma$. The isomorphism maps the loop to this element. 

\begin{prop}\label{prop:absurd}
  Let $G$ be a quasi connected \lcsc unimodular group. Let $\Gamma<G$ be a torsion free lattice, let $U\subseteq K\subseteq G$ be neighborhoods of the identity such that $K_0U=U$, $U=U^{-1}$, $\rho(U^2)$ contractible and $K\supset U^4$ compact. Let $Y\subseteq (G/\Gamma)_{U^2}$ be a $U$-cross section and assume that $\id_G\in Y$. Then $\Gamma^K$ contains the image of $\pi_1(\rho^\Gamma(YU^2),\id_G)$ inside $\pi_1(M/\Gamma,\id_G)=\Gamma$ through the map induced by the inclusion $\rho^\Gamma(YU^2)\hookrightarrow M/\Gamma$. 
\end{prop}
\begin{proof}
  Consider a loop $\ell$ in $\rho^\Gamma(YU^2)$. Since $\rho(U^2)$ is contractible and $Y\subseteq (G/\Gamma)_{U^2}$ we have that $\rho^\Gamma(yU^2)$ is contractible for every $y\in Y$. Observe also that $\rho^\Gamma$ restricted to $Y$ and $\rho$ restricted to $\widetilde Y\coloneqq \pi^{-1}(Y)\subseteq G$ are injective. By compactness there are finitely $\rho^\Gamma(y_1),\ldots,\rho^\Gamma(y_h)\in \rho^\Gamma(Y)$ such that $\rho^\Gamma(y_1)=\rho^\Gamma(y_n)=\id_G$ and $\ell$ is contained in $\cup_i \rho^\Gamma(y_i)U^2$ and $\rho^\Gamma(y_{i+1})\in U^4\rho^\Gamma(y_i)$. Consider a lifting $p$ of $\ell$ inside $M$ whose starting point is the identity. Then there are $\rho(\tilde y_i) \in \rho(\widetilde Y)$ which lifts the elements $\rho^\Gamma(y_i)$ such that $\rho^\Gamma(\tilde y_1)=\id_G$ and such that $p$ is contained in $\cup_i \rho(\tilde y_i)U^2$ and $\rho(\tilde y_{i+1})\in U^4\rho(\tilde y_i)$. This implies also that $\tilde y_{i+1}\in U^4\tilde y_i$. Finally observe that the image of $\ell$ seen as an element of $\pi_1(\rho^\Gamma(YU^2))$ inside $\pi_1(M/\Gamma)=\Gamma$ is the element $\rho(\tilde y_h)\in \Gamma$. Therefore we can use Lemma \ref{lem:notproven} to complete the proof.
\end{proof}

Combining the above proposition, Theorem \ref{thm:costgenp} and a result of Gelander \cite{gel2004} we obtain the following. 

\begin{thm}\label{thm:costgela}
  Let $G$ be a semisimple Lie group without compact factors and fix a Haar measure on $G$. Let $\{\Gamma_n\}_n$ be a Farber sequence of torsion free lattices. Then \[\cost\left(G\curvearrowright [G/\Gamma_n]_\ul^{R}\right)\geq 1+\lim_\ul \frac{\rank(\Gamma_n)-1}{\covol(\Gamma_n)}\geq \cost(G).\]
\end{thm}
\begin{proof}
  Gelander proved in Theorem 1.5, Lemma 4.1, Proposition 4.8, Theorem 7.4 and Section 9 of \cite{gel2004}, that for every such a Lie group $G$ there exists a neighborhood of the identity $U$ (which is just a ball of some radius) for which for every lattice $\Gamma$ there exists $Y\subseteq (G/\Gamma)_{U^2}$ satisfying that the map induced by the inclusion from $\pi_1(\rho^\Gamma(YU^2))$ to $\pi_1(M/\Gamma)=\Gamma$ is surjective. Therefore given a Farber sequence $\{\Gamma_n\}_n$ of lattices we can find cross sections $Y_n$ for which $\Gamma_n^K=\Gamma_n$ for every $n$ and hence Theorem \ref{thm:costgenp} allows us to conclude the proof.
\end{proof}

\section{\texorpdfstring{$\ell^2$}{Lg}-Betti numbers}

In the last section, we studied the asymptotics of the minimal number of generators of lattices of unimodular, compactly generated \lcsc groups. We showed that the limit is related to the cost of the group. In this section, we will study the similar question for the Betti numbers of the lattices. Since we will work with higher homological invariants, the hypothesis that $G$ is compactly generated is not sufficient anymore. We will need that $G$ acts cocompactly and properly on a \textit{nice} topological space.

We say that an open cover of a topological space is a \defin{good cover} if the finite intersections of open sets of the cover are either empty or contractible. The \defin{nerve} of an open cover is the simplicial complex whose vertices are the elements of the open cover and whose $k$-simplices correspond to the intersections of $k+1$ open sets. It is well known, see for example Corollary $4G.3$ of \cite{hat2002}, that the nerve of a good cover of a paracompact space is homotopic to the space itself. 

\begin{dfn}\label{dfn:wellcovered}
  A \defin{well-covered} $G$-\defin{space} is a topological space $M$ on which $G$ acts continuously on the right and such that 
  \begin{itemize}
  \item the action is \textit{proper}, that is the induced map $G\times M\rightarrow M\times M$ given by $(g,m)\mapsto (mg^{-1},m)$ is proper\footnote{A map is proper if the preimage of any compact set is compact.};
  \item the action is \textit{cocompact}, that is the quotient $M/G$ is a compact and Hausdorff space;
  \item there exists a $G$-\textit{invariant good cover} $\co$ made of open and relatively compact sets;
  \item there are \textit{fine good covers}, that is there are compact subsets $A_M\subseteq M$ and $A_G\subseteq G$ such that for every symmetric neighborhood of the identity $W$, there are $O_1,\ldots,O_m\in\co$ satisfying $\cup_iO_i\subseteq A_M$, $(\cup_i O_i)G=M$ and whenever we have $O_ig_1\cap O_j\neq \emptyset$ and $O_ig_2\cap O_j\neq\emptyset$ for some $1\leq i,j\leq m$ and $g_1,g_2\in G$, then $g_1g_2^{-1}\in WK$ for some compact subgroup $K\subseteq A_G$.
  \end{itemize}
\end{dfn}

Notable examples of well-covered $G$-spaces are given by isometric, proper and cocompact actions of \lcsc groups on Riemannian manifolds. Indeed the set of balls of arbitrarily small radius form the fine good coverings, see Lemma \ref{lem: riemann well-covered}. Another example is given by proper and cocompact simplicial actions on simplicial spaces and in this case the fine good coverings are given by the \textit{stars} of barycentric subdivisions, see Lemma \ref{lem: simplicial well-covered}. The last condition of Definition \ref{dfn:wellcovered} is inspired by the above two cases, it is a key feature that both open covers share and it is what is needed in order to see $M$ and $M/\Gamma$ as $\Gamma$-simplicial complexes in a coherent way, see \ref{sect:fromGtoGamma}.

For a topological space $Q$ we will denote by $b_i(Q)$ the $i$-th Betti number, that is the dimension of the $i$-th homology group (with rational coefficients). From now on, we will also use the terminology introduced in Definition \ref{dfn:Farberthinlatt}. The main theorem of this section is the following. 

\begin{thm}\label{thm:mainl2}
  Let $G$ be a unimodular \lcsc compactly generated group and fix a Haar measure $\haar$ on $G$. Let $M$ be a contractible well-covered $G$-space. Let $\{\Gamma_n\}_n$ be a Farber sequence of torsion free lattices of $G$. Then we have \[\beta_i(G)\leq \liminf_n \frac{b_i(M/\Gamma_n)}{\covol(\Gamma_n)},\] and if the sequence $\{\Gamma_n\}$ is nowhere thin, then \[\beta_i(G)= \lim_n \frac{b_i(M/\Gamma_n)}{\covol(\Gamma_n)}.\]
\end{thm}

By $\beta_i(G)$ we mean the $i$-th $\ell^2$-Betti number of $G$ as defined in \cite{pet2013}. This invariant is equal to the quantity  $\beta_i(\Gamma)/\covol(\Gamma)$ for any lattice $\Gamma$ in $G$, see \cite{kye2015}. In this work, we will not directly work with the $\ell^2$-Betti numbers of $G$, but with the $\ell^2$-Betti numbers of the cross-sections of its actions. Indeed, given any free p.m.p.\ action of $G$ on $(X,\mu)$ and given any cross-section $Y\subseteq X$, then $\beta_i(G)=\beta_i(\rel_Y)/\covol(Y)$, see \cite{kye2015} and the discussions in Section \ref{section: reminders about ell 2 betti numbers}.

It is not hard to observe that if $\Gamma$ is a cocompact torsion free lattice of $G$, then $b_i(M/\Gamma)=b_i(\Gamma)$, see Corollary \ref{crl: beta i M Gamma = beta i Gamma}. In Section \ref{subsection:examples of well-covered} we will give examples of unimodular \lcsc compactly generated groups which admit contractible well-covered spaces. In particular we get the following.

\begin{crl}\label{crl:excomp}
  Let $G$ be a unimodular \lcsc compactly generated group and let us fix a Haar measure $\haar$ on $G$. Let $\{\Gamma_n\}_n$ be a Farber sequence of cocompact torsion free lattices of $G$. Assume that one of the following conditions holds. 
  \begin{itemize}
  \item The group $G$ is almost connected.
  \item The group $G$ is totally disconnected and it acts cocompactly and properly on a contractible simplicial complex.
  \item The quotient $G/G_0$ satisfies the previous condition, where $G_0\leq G$ denotes the connected component of the identity
  \end{itemize}
  Then \[\beta_i(G)\leq \liminf_n \frac{b_i(\Gamma_n)}{\covol(\Gamma_n)}\] and if the sequence $\{\Gamma_n\}$ is nowhere thin, then \[\beta_i(G)= \lim_n \frac{b_i(\Gamma_n)}{\covol(\Gamma_n)}.\] 
\end{crl}

Before stating the next theorem we will need the following notation: for an inclusion of topological spaces $P\subseteq Q$, we will denote by $\nabla_i(P,Q)$ the dimension of the image of the map induced by the inclusion at the level of homologies $H_i(P)\rightarrow H_i(Q)$. 

Let us assume that the well-covered $G$-space $M$ admits a proper and compatible metric $d_M$ such that $G$ acts isometrically on $(M,d_M)$. This is the case for example if $M$ is a Riemannian manifold and $d_M$ is the geodesic distance. Let $\Gamma<G$ be a lattice and $\eps$ a real positive number. We define \[(M)^\Gamma_\eps\coloneqq\left\{p\in M\colon\forall \gamma\in\Gamma\setminus\{\id_\Gamma\},\ d_M(p\gamma,p)>\eps\right\}\] and we will set $(M/\Gamma)_\eps\coloneqq(M)^\Gamma_\eps/\Gamma$. We will prove the following theorem. 

\begin{thm}\label{thm:epsneigh}
  Let $G$ be a unimodular \lcsc compactly generated group and let us fix a Haar measure $\haar$ on $G$. Let $M$ be a contractible well-covered $G$-space and assume that $M$ admits a proper and compatible metric $d_M$ such that $G$ acts isometrically on $(M,d_M)$. Let $\{\Gamma_n\}_n$ be a Farber sequence of torsion free lattices of $\Gamma$. Then for every $\eps>0$ we have that \[\beta_i(G)=\lim_n \frac{\nabla_i((M/\Gamma_n)_\eps,M/\Gamma_n)}{\covol(\Gamma_n)}.\]
\end{thm}

Let us give a concrete example. Let $G$ be a Lie group, let us denote by $K_0\leq G$ a maximal compact subgroup. Take a Riemannian left $K_0$-invariant and right $G$-invariant metric $d_G$ on $G$. Set $M\coloneqq K_0\backslash G$ and denote by $d_M$ the metric induced by $d_G$. Then clearly the action by translation of $G$ on the right is isometric and $M$ is a well-covered $G$-space by Lemma \ref{lem: riemann well-covered}. 

Theorem \ref{thm:epsneigh} becomes extremely easy to understand whenever we have that $(M/\Gamma_n)_\eps$ already contains all the homological information of the space, that is when $\nabla_i((M/\Gamma_n)_\eps,M/\Gamma_n)=b_i(M/\Gamma_n)$, or more generally when the dimension of the homology of $M/\Gamma_n\setminus (M/\Gamma_n)_\eps$ is sublinear with respect to the covolume. This, in some cases, is an immediate consequence of the thin-thick decomposition, \cite[Theorem 4.5.6]{thu1997}, the work of Gromov, \cite{gro1985} and the more recent work of Gelander, \cite{gel2004}. In this way we are able to get a result similar to Theorem 1.8 of \cite{abe2017} and the main result of \cite{abbg2018}.

\begin{crl}\label{crl:gela}
  Assume that $G$ is a connected semisimple Lie group and let us fix a Haar measure $\haar$ on $G$. Let $\{\Gamma_n\}_n$ be a Farber sequence of torsion free lattices of $G$.
  \begin{itemize}
  \item It always follows that
    \[\beta_1(G)= \lim_n \frac{b_1(\Gamma_n)}{\covol(\Gamma_n)}.\]
  \item If $G$ has higher rank and $\Gamma_n$ is not cocompact for every $n$, then for every $i$ \[\beta_i(G)= \lim_n \frac{b_i(\Gamma_n)}{\covol(\Gamma_n)}.\]
  \item If $G$ has rank $1$ and the associated symmetric space $M\coloneqq K_0\backslash G$ has dimension $d\geq 4$, then for every $i\neq d-1$ \[\beta_i(G)= \lim_n \frac{b_i(\Gamma_n)}{\covol(\Gamma_n)}.\]
  If $\Gamma_n$ is cocompact for every $n$, then the above equation holds for every $i$.
  \end{itemize}
\end{crl}
\begin{proof}
  In all cases we have that $\nabla_i((M/\Gamma_n)_\eps,M/\Gamma_n)=b_i(M/\Gamma_n)$. The first two points follow from the work of Gelander \cite{gel2004} exactly as in the proof of Theorem \ref{thm:costgela}. The third point follows from the thick-thin decomposition, see \cite[Theorem 4.5.6]{thu1997} for the hyperbolic case, and the Mayer-Vietoris exact sequence. 
  
  Indeed let us fix a lattice $\Gamma$ of a rank $1$ Lie group. First note that the result is trivial for $i=0$ or $i\geq d$. For $i=1$ if follows from the fist part and Poincar\'{e} duality implies that the statement holds for $i=d-1$ whenever $\Gamma_n$ is cocompact. Therefore let us fix $1<i\leq d-2$. The thick-thin decomposition states that there exists $\eps>0$ independent of $n$, such that $M/\Gamma_n\setminus (M/\Gamma_n)_\eps$ is composed of a disjoint union of neighborhoods of closed geodesics and cusps. Fix $n$ and let us denote by $B$ and $C$ the open set obtain as the union of the neighborhoods of the closed geodesics and cusps respectively, so that $M/\Gamma_n=(M/\Gamma_n)_\eps\cup B\cup C$.

  First remark that the cusps are homeomorphic to products of manifolds with the reals. In particular, it is easy to check that $\nabla_i((M/\Gamma_n)_\eps\cup B,M/\Gamma_n)=b_i(M/\Gamma_n)$. On the other hand, let us observe that each connected component of $B$ is a disk bundle over a circle. Therefore $(M/\Gamma_n)_\eps\cap \bar B$ is homotopic to a sphere bundle over a circle and its homology vanishes in every dimension not equal to $0,1,d-2$ and $d-1$. Since $i>1$, we clearly have that $H_i(B)=0$. Therefore Mayer-Vietoris tells us that \[H_i((M/\Gamma_n)_\eps)\rightarrow H_i((M/\Gamma_n)_\eps\cup B)\rightarrow H_{i-1}((M/\Gamma_n)_\eps\cap \bar B)\rightarrow H_{i-1}((M/\Gamma_n)_\eps)\oplus H_{i-1}(\bar B).\]
  Now if $2<i\leq d-2$, then $H_{i-1}((M/\Gamma_n)_\eps\cap \bar B)=0$ and hence the map $H_i((M/\Gamma_n)_\eps)\rightarrow H_i((M/\Gamma_n)_\eps\cup B)$ is surjective. For $i=2$, observe that the map induced by the inclusion $H_1((M/\Gamma_n)_\eps\cap \bar B)\rightarrow H_{1}(\bar B)$ is an isomorphism and in particular is injective. Therefore, the exactness of the sequence implies that $H_2((M/\Gamma_n)_\eps)\rightarrow H_2((M/\Gamma_n)_\eps\cup B)$ is surjective and the proof of the corollary is completed.
\end{proof}

  We do not know whether for a rank $1$ Lie group we have $\nabla_{d-1}((M/\Gamma_n)_\eps,M/\Gamma_n)=b_{d-1}(M/\Gamma_n)$. We do not see any easy way to deduce Corollary \ref{crl:gela} for non-cocompact lattices in rank $1$ Lie groups in dimension $d-1$ from Theorem \ref{thm:epsneigh} and the thick-thin decomposition. However one could obtain the result by proceeding as in \cite[Proposition 3.1]{abbg2018}: there are very few short geodesics compared to the volume. 

\subsection{Reminders about $\ell^2$-Betti numbers for equivalence relations and cross-sections}\label{section: reminders about ell 2 betti numbers}

Let us briefly recall the notion of $\ell^2$-Betti number of \pmp equivalence relations with countable classes defined in \cite{gab2002}. Let $\rel$ be a \pmp equivalence relation over the probability space $(X,\mu)$. A $\rel$-\textit{simplicial complex} is a field of simplicial complexes $\Sigma$ over $X$ on which $\rel$ acts smoothly and commuting with the boundary operators, see \cite[D\'efinition 2.6]{gab2002} or Section 5 of \cite{CGS} in the setting of general probability spaces. For $x\in X$ we will denote by $\Sigma_x$ the simplicial complex ``over'' the point $x\in X$. We say that such a complex is \textit{uniformly bounded} if there exists a $D$ such that for almost every $x\in X$, the $1$-skeleton of $\Sigma$ is a graph of degree bounded by $D$. Note that this implies that the field of simplicial complexes is uniformly bounded in every dimension.

One can define the (simplicial) chain complex associated with $\Sigma$ and complete it to a Hilbert space by using the measure $\mu$ on $X$. If the simplicial complex is uniformly bounded, then the boundary operators are bounded and one can define the (reduced) homology as usual. Since the boundary operators commute with the action of $\rel$, the equivalence relation still acts on the homology and we will denote by $\beta_i(\Sigma)$ the $\rel$-dimension of the $i$-th homology group. Gaboriau proved that the value $\beta_i(\Sigma)$ only depends on the \textit{homotopy-class} of $\Sigma$, \cite[Corollaire 4.9]{gab2002}. In particular, whenever $\Sigma$ is a $\rel$-simplicial complex such that $\Sigma_x$ is contractible for almost every $x\in X$, we will set $\beta_i(\rel)\coloneqq \beta_i(\Sigma)$. This value does not depend on the chosen contractible $\rel$-simplicial complex $\Sigma$. Moreover if $\rel$ is the orbit equivalence relation induced by an action of a countable group $\Gamma$ on a probability space $(X,\mu)$ and $\Gamma$ acts freely on the contractible simplicial complex $S$, then one can form the $\rel$-simplicial complex $\Sigma\coloneqq S\times X$ and one obtains $\beta_i(\Gamma)=\beta_i(\rel)$ \cite[Th\'eor\`eme 3.11]{gab2002}. 

Let us take a unimodular \lcsc compactly generated group $G$ and let us fix a Haar measure. Suppose we have a free action of $G$ on the probability measure space $(X,\mu)$ preserving the measure. Consider an external cross section $(Y,\Psi)$. Denote as always by $\rel_Y$ the cross equivalence relation. Then $\beta_i(G)=\beta_i(\rel_Y)/\covol(Y)$, so that the ratio only depends on $G$ and its Haar measure, \cite{kye2015}. Note that the fact that the ratio $\beta_i(\rel_Y)/\covol(Y)$ only depends on $G$ and its Haar measure can be easily proven with the results in \cite{gab2002} and we will not explicitly use any of the results in \cite{kye2015}.

\subsection{From well-covered $G$-spaces to $\Gamma$-simplicial complexes}\label{sect:fromGtoGamma}

Let $G$ be a unimodular \lcsc compactly generated group and let $M$ be a well-covered $G$-space. Since $M$ has a $G$-invariant good cover, it is homotopic to a simplicial complex on which $G$ acts. If $G$ is not countable, the nerve of the good cover is a simplicial complex with in general uncountably many simplices. However, whenever we fix a countable subgroup $\Gamma<G$, we can take a countable $\Gamma$-invariant good cover of $M$ and hence obtain an action of $\Gamma$ on a simplicial complex with countably many cells. If $\Gamma<G$ is a lattice, then we obtain that $M$ is homotopic to a simplicial complex on which $\Gamma$ acts freely.

\begin{lem}\label{lem: invariant free Gamma cover}
  Let $G$ be a unimodular \lcsc compactly generated group and let $M$ be a well-covered $G$-space. Denote its $G$-invariant good cover by $\mathcal O$. Let $\Gamma<G$ be a cocompact torsion free lattice. Then there are $O_1,\ldots,O_m\in\mathcal O$ such that 
  \begin{itemize}
      \item the cover $\mathcal O_\Gamma\coloneqq \{O_i\gamma\}_{1\leq i\leq m,\gamma \in \Gamma}$ is a good cover of $M$;
      \item the quotient cover $\mathcal O_\Gamma/\Gamma$ on $M/\Gamma$ is a good cover.
  \end{itemize}
\end{lem}

Lemma \ref{lem: invariant free Gamma cover} tells us that well-covered $G$-spaces behave in a very good manner when we see them as $\Gamma$-spaces. Indeed $M$ is not only homotopic to a bounded simplicial complex on which $\Gamma$ acts cocompactly, but also that taking the nerve and the quotient by the action of $\Gamma$ commute. We will not use Lemma \ref{lem: invariant free Gamma cover} explicitly, but the methods of its proof will be essential for proving Theorem \ref{thm:mainl2}.

 We will need the following well-known fact. 

\begin{lem}\label{lem:Pdisc}
  Let $G$ be a unimodular \lcsc compactly generated group.
  For every neighborhood of the identity $V\subseteq G$ and compact subset $A\subseteq G$ there exists a symmetric neighborhood of the identity $W$, such that for every lattice $\Gamma$ satisfying $\Gamma\cap V=\{\id_G\}$, we have that $\Gamma\cap WK$ consists of elements of finite order, for every compact subgroup $K$ contained in $A$. 
\end{lem}
\begin{proof}
  Let $W_1\subseteq G$ be an open subset such that $W_1W_1^{-1}\subseteq V$. Consider $\Gamma$ such that $\Gamma\cap V=\{\id_G\}$. Remark that for every $g\in G$, we have that $|\Gamma\cap W_1g|\leq 1$. Indeed if $\gamma_1,\gamma_2\in \Gamma\cap W_1g$, then $\gamma_1\gamma_2^{-1}\in \Gamma\cap W_1W_1^{-1}=\{\id\}$. Let $N$ be the maximal cardinality of a $W_1$-separated subset of $W_1A$. Observe that $N\leq \haar(W_1^2A)/\haar(W_1)$ is finite. Consider a symmetric neighborhood of the identity $W_2\subseteq G$ such that $W_2^{N}\subseteq W_1$. Since $A$ is compact there is a symmetric neighborhood of the identity $W\subseteq G$ such that $gWg^{-1}\subseteq W_2$ for every $g\in A$. Let $K\subseteq A$ be a compact subgroup and consider $\gamma\in \Gamma\cap WK$. Let $n\leq N$ and observe that $\gamma^n\in W_1K$. Indeed if $\gamma=wk$ with $w\in W$ and $k\in K$, then $\gamma^n=(wk)^n=w(kwk^{-1})(k^2wk^{-2})\ldots (k^{n-1}wk^{-(n-1)})k^n\in W_2^nK\subseteq W_1K.$ Then $\id_\Gamma,\gamma,\gamma^2,\ldots,\gamma^N$ are $N+1$ elements in $\Gamma\cap W_1K\subseteq W_1A$. But since $|\Gamma\cap W_1g|\leq 1$ for every $g\in G$, these elements are $W_1$-separated and by definition of $N$ we must have that for some $n\leq N$, we have that $\gamma^n=\id_\Gamma$ as claimed.
\end{proof}

\begin{proof}[Proof of Lemma \ref{lem: invariant free Gamma cover}]
 There is a neighborhood of the identity $V\subseteq G$ such that $V\cap g\Gamma g^{-1}=\{\id_G\}$ for every $g\in G$. Let us fix the compact subset $A_G\subseteq G$ as in the last condition of well-covered $G$-spaces, Definition \ref{dfn:wellcovered}. Let $W$ be as in Lemma \ref{lem:Pdisc} with respect to the neighborhood $V$ and the compact subset $A_G$. Then since $M$ is a well-covered $G$-space, there are $O_1,\ldots,O_m\in\co$ satisfying $(\cup_i O_i)G=M$ and whenever we have $O_ig_1\cap O_j\neq \emptyset$ and $O_ig_2\cap O_j\neq\emptyset$ for some $1\leq i,j\leq m$ and $g_1,g_2\in G$, then $g_1g_2^{-1}\in WK$ for some compact subgroup $K\subseteq A_G$. 
By cocompactness of $G$ on $M$ and $\Gamma$, there is a finite subset $F\subseteq G$ such that $\cup_{1\leq i\leq m}\cup_{f\in F}O_if\Gamma=M$. Assume now that for some $1\leq i,j\leq m$, $\gamma_1,\gamma_2\in \Gamma$ and $f_1,f_2\in F$ we have that 
\[O_if_1\gamma_1\cap O_jf_2\neq \emptyset\text{ and }O_if_1\gamma_2\cap O_jf_2\neq\emptyset.\]
Then by assumption $f_1\gamma_1\gamma_2^{-1}f_1^{-1}\in WK$ for some compact subgroup $K$ contained in $A_G$. Remark that $\Gamma$ is torsion free and $f_1\Gamma f_1\cap V=\{\id_G\}$, therefore Lemma \ref{lem:Pdisc} yields that $\gamma_1=\gamma_2$. 
In particular, if we take $i=j$ and $f_1=f_2$, we obtain that the action of $\Gamma$ on the cover $\{O_if\gamma\}_{i,f,\gamma}$ is free. 
Finally, observe that since $O_if\gamma\cap O_if$ if and only if $\gamma=\id_\Gamma$, we have that $O_if$ seen inside $M/\Gamma$ is contractible. Similarly, since given $1\leq i,j\leq m$ and $f_1,f_2\in F$, the sets $O_if_1$ and $O_jf_2\gamma$ meets for at most one $\gamma\in\Gamma$, their intersection in the quotient $M/\Gamma$ is also contractible. Since this is true for arbitrary finite intersections, we obtain that the sets $\{O_if\}_{1\leq i\leq m,f\in F}$ induce a good cover of $M/\Gamma$.    
\end{proof}

Let us also remark that Lemma \ref{lem:Pdisc} implies that the action of any lattice on a well-covered $G$-space is properly discontinuous. In the next corollary we will denote by $\beta_i(M,\Gamma)$ the $i$-th $\ell^2$-Betti number of $M$ seen as a $\Gamma$-space (or simplicial complex).

\begin{crl}\label{crl: beta i M Gamma = beta i Gamma}
  Let $G$ be a unimodular \lcsc compactly generated group and let $M$ be a well-covered $G$-space. If $\Gamma$ is a torsion free lattice of $G$, then the action of $\Gamma$ on $M$ is free and properly discontinuous. In particular if $M$ is contractible, then $\beta_i(M,\Gamma)=\beta_i(\Gamma)$. 
\end{crl}

\subsection{Proof of Theorem \ref{thm:mainl2} and Theorem \ref{thm:epsneigh}}

Let $G$ be a unimodular \lcsc compactly generated group, let $V\subseteq G$ be any neighborhood of the identity and let $M$ be a well-covered $G$-space. Let us fix the compact subset $A_M\subseteq M$ as in the last condition of well-covered $G$-spaces, Definition \ref{dfn:wellcovered}. Recall that $A_MG=M$.

 Given any neighborhood of the identity $U\subseteq G$ and a lattice $\Gamma<G$, let us set \[(G)_U^\Gamma\coloneqq \{g\in G\colon g\Gamma g^{-1}\cap U=\{\id_G\}\}.\]
 Remark that $(G)_U^\Gamma/\Gamma=(G/\Gamma)_U$ as defined in Definition \ref{dfn:thickness}. 

 Recall that given an inclusion of topological spaces $P\subseteq Q$, we denote by $\nabla_i(P,Q)$ the dimension of the image of the map induced by the inclusion at the level of homologies $H_i(P)\rightarrow H_i(Q)$.
 
 \begin{thm}\label{thm:maintechnical}
   Let $G$ be a unimodular \lcsc compactly generated group and fix a Haar measure $\haar$ on $G$. Let $M$ be a contractible well-covered $G$-space. Let $\{\Gamma_n\}_n$ be a Farber sequence of torsion free lattices of $G$. Then for any neighborhood of the identity $V$
   \[\beta_i(G)=\lim_\ul\frac{\nabla_i\left(A_M(G)_{V}^{\Gamma_n}/\Gamma_n,M/\Gamma_n\right)}{\covol(\Gamma_n)}.\]
 \end{thm}

We will prove Theorem \ref{thm:maintechnical} in the next section, Section \ref{section: proof of thm maintechnical}. Let us show first how it implies both Theorem \ref{thm:mainl2} and Theorem \ref{thm:epsneigh}. We start with Theorem \ref{thm:mainl2}. For the reader's convenience, let us repeat the statement. 

\begin{thm*}[Theorem \ref{thm:mainl2}]
  Let $G$ be a unimodular \lcsc compactly generated group and fix a Haar measure $\haar$ on $G$. Let $M$ be a contractible well-covered $G$-space. Let $\{\Gamma_n\}_n$ be a Farber sequence of torsion free lattices of $G$. Then we have \[\beta_i(G)\leq \liminf_n \frac{b_i(M/\Gamma_n)}{\covol(\Gamma_n)},\] and if the sequence $\{\Gamma_n\}$ is nowhere thin, then \[\beta_i(G)= \lim_n \frac{b_i(M/\Gamma_n)}{\covol(\Gamma_n)}.\]
\end{thm*}
\begin{proof}
Clearly we have that \[\nabla_i\left(A_M(G)_{V}^{\Gamma_n}/\Gamma_n,M/\Gamma_n\right)\leq b_i(M/\Gamma_n).\]
Therefore Theorem \ref{thm:maintechnical} yields \[\beta_i(G)\leq \lim_\ul\frac{b_i(M/\Gamma_n)}{\covol(\Gamma_n)}\] and since this is true for every ultrafilter, we get 
\[\beta_i(G)\leq \liminf_n\frac{b_i(M/\Gamma_n)}{\covol(\Gamma_n)}.\]

Suppose now that $\{\Gamma_n\}_n$ is nowhere thin. Then there is a neighborhood of the identity $U\subseteq G$ such that $(G/\Gamma_n)_U=G/\Gamma_n$ for every $n$. This implies that $(G)_U^{\Gamma_n}=G$ and hence Theorem \ref{thm:maintechnical} gives us that  \[\beta_i(G)=\lim_\ul\frac{\nabla_i\left(A_M(G)_{U}^{\Gamma_n}/\Gamma_n,M/\Gamma_n\right)}{\covol(\Gamma_n)}=\lim_\ul\frac{\nabla_i(M/\Gamma_n,M/\Gamma_n)}{\covol(\Gamma_n)}=\lim_\ul\frac{b_i(M/\Gamma_n)}{\covol(\Gamma_n)}.\qedhere\]
\end{proof}

Before showing how Theorem \ref{thm:maintechnical} implies Theorem \ref{thm:epsneigh} let us note the following easy fact: if $P_1\subseteq P_2\subseteq P_3$ are topological spaces, then \[\nabla_i(P_1,P_3)\leq \nabla_i(P_2,P_3).\]

Let us recall the notation \[(M)^\Gamma_\eps=\left\{p\in M\colon\forall \gamma\in\Gamma\setminus\{\id_\Gamma\},\ d_M(p\gamma,p)>\eps\right\}\]  and $(M/\Gamma)_\eps=(M)^\Gamma_\eps/\Gamma$. 

\begin{thm*}[Theorem \ref{thm:epsneigh}]
  Let $G$ be a unimodular \lcsc compactly generated group and let us fix a Haar measure on $G$. Let $M$ be a contractible well-covered $G$-space and assume that $M$ admits a proper and compatible metric $d_M$ such that $G$ acts isometrically on $(M,d_M)$. Let $\{\Gamma_n\}_n$ be a Farber sequence of torsion free lattices of $\Gamma$. Then for every $\eps>0$ we have that \[\beta_i(G)=\lim_n \frac{\nabla_i((M/\Gamma_n)_\eps,M/\Gamma_n)}{\covol(\Gamma_n)}.\]
\end{thm*}
\begin{proof}
   Let $\Gamma<G$ be any lattice and fix $\eps>0$. Since the action is proper, there is $V\subseteq G$ open and precompact neighborhood of the identity such that $d_M(pg,p)> \eps$ whenever $g\notin V$ and $p\in A_M$.

   We claim that \[(M)_\eps^\Gamma\supseteq A_M(G)_{V}^\Gamma.\]
  Indeed fix $p\in A_M$ and $g\in (G)_{V}^\Gamma$. By assumption $g\Gamma g^{-1}\cap V=\{\id_G\}$, so for every $\gamma\in\Gamma$ not the identity \[d_M(pg\gamma,pg)=d_M(pg\gamma g^{-1},p)>\eps\]
  that is $pg\in(M)_\eps^\Gamma$ as claimed. 
  
  Let now $U\subseteq G$ be a precompact neighborhood of the identity such that $d_M(pu,p)\leq\eps$ for every $u\in U$ and $p\in A_M$. Then we claim that \[(M)_\eps^\Gamma\subseteq A_M(G)_{U}^\Gamma.\]
  For this, consider $p\in M\setminus A_M(G)_{U}^\Gamma$. Then there is $g\in G\setminus (G)_U^\Gamma$ such that $pg^{-1}\in A_M$. 
  Let $u\neq \id_G$ be such that $g^{-1}ug\in \Gamma$. Then by assumption on $U$, we have that \[d_M(pg^{-1}ug,p)=d_M(pg^{-1}u,pg^{-1})\leq \eps,\]
  which implies that $p\notin (M)_{\eps}^\Gamma$ and the claim is proven. 
  
  We can now finish the proof of the theorem, 
  \begin{align*}
      \beta_i(G)=\lim_\ul\frac{\nabla_i\left(A_M(G)_{V}^{\Gamma_n}/\Gamma_n,M/\Gamma_n\right)}{\covol(\Gamma_n)}\leq & \lim_\ul\frac{\nabla_i\left((M)_\eps^{\Gamma_n}/\Gamma_n,M/\Gamma_n\right)}{\covol(\Gamma_n)}\leq\\ &
  \lim_\ul\frac{\nabla_i\left(A_M(G)_{U}^{\Gamma_n}/\Gamma_n,M/\Gamma_n\right)}{\covol(\Gamma_n)}=\beta_i(G).\qedhere
  \end{align*}
\end{proof}

\subsection{Proof of Theorem \ref{thm:maintechnical}}\label{section: proof of thm maintechnical}

Let us once again repeat the setting and prove some preliminary lemmas. 

Let $G$ be a unimodular \lcsc compactly generated group and let $V\subseteq G$ be any symmetric neighborhood of the identity. 
Let $M$ be a well-covered $G$-space. Let us take the compact subsets $A_M\subseteq M$ and $A_G\subseteq G$ as in the last condition of well-covered $G$-spaces, Definition \ref{dfn:wellcovered}. 

Since the action of $G$ on $M$ is proper, there is a compact subset $K_0'$ such that whenever $A_Mg\cap A_M\neq\emptyset$, then $g\in K_0'$. Again, since the action is proper, there is $K_0$ such that whenever $A_MK_0'g\cap A_MK_0'\neq\emptyset$, then $g\in K_0$.

Let $W$ as in Lemma \ref{lem:Pdisc} with respect to the neighborhood $V$ and the compact subset $A_G$. Then since $M$ is a well-covered $G$-space, there are $O_1,\ldots,O_m\in\co$ satisfying $\cup_i O_i\subseteq A_M$, $(\cup_i O_i)G=M$ and whenever we have $O_ig_1\cap O_j\neq \emptyset$ and $O_ig_2\cap O_j\neq\emptyset$ for some $1\leq i,j\leq m$ and $g_1,g_2\in G$, then $g_1g_2^{-1}\in WK$ for some compact subgroup $K\subseteq A_G$. Since  for every $i\in\{1,\ldots,m\}$ the set $O_i$ is relatively compact, there exists a finite subset $S\subseteq K_0$ such that \[\cup_{i}\bar O_i\subseteq \cup_{i}O_iS.\]

Let $\ul$ be an ultrafilter and consider a $\ul$-Farber sequence of lattices $\{\Gamma_n\}_n$ of $G$ and set $X_n\coloneqq G/\Gamma_n$. Fix a Haar measure $\haar$ on $G$ and denote by $\mu_n$ the induced probability measure on $X_n$. Recall that given any neighborhood of the identity $U\subseteq G$, we have \[(X_n)_U=\{x\in X_n\colon ux\neq x,\ \forall U\setminus \{\id_G\}\}.\]
 Since $\{\Gamma_n\}_n$ is $\ul$-Farber, we have that for every precompact open neighborhood of the identity $U\subseteq G$ one has that $\lim_\ul\mu_n((X_n)_U)=1$. 

 The set $\cup_i\bar O_i\bar V^2$ is compact, clearly $\cup_i \bar O_i\bar V^2\subseteq \cup_iO_iSV^3$, therefore there is a finite subset $F\subseteq SV^3\subseteq K_0V^3$ such that \[\cup_iO_iV^2\subseteq \cup_i\bar O_i\bar V^2\subseteq \cup_i O_iF.\]
 
Our first step is to fix the sequence of cross-sections $Y_n$. Set $V_1\coloneqq V^{-3}K_0^{-1}VK_0V^3.$

\begin{lem}\label{lem: cross-section}
 For every $n$, there is a finite subset $Y_n\subseteq (X_n)_{V_1}$ containing the class of the identity, satisfying the following conditions
 \begin{itemize}
     \item $Y_n$ is a $V$-cross section;
     \item $V^2Y_n\supseteq (X_n)_{V_1}$;
     \item $FY_n\subseteq (X_n)_{V}$.
 \end{itemize}
\end{lem}
\begin{proof}
  Let $Y_n\subseteq (X_n)_{V_1}$ be a maximal $V$-separated subset. Since $V_1\supseteq V^2$, we can apply Lemma \ref{lem:filling} to get that $Y_n$ is a cross-section and by maximality $V^2Y_n\supseteq (X_n)_{V_1}$. Finally for every $f\in F$, $y\in Y_n\subseteq (X_n)_{V_1}$, and $v\in V$, we have that $f^{-1}vf\in V_1$ and therefore $vfy=ff^{-1}vfy\neq fy$
  so that $fy\in (X_n)_V$ as claimed. 
\end{proof}

Since the action of $G$ on $M$ is proper, there is a compact subset $B_0\subseteq G$, such that for every $g\in G\setminus B_0$ \[\left(\cup_iO_iF\right)\cap \left(\cup_iO_iF\right)g=\emptyset.\]

The group $G$ acts on the regular ultraproduct $X_\ul^{R}$ of the sequence $\{X_n\}_n$ and Theorem \ref{thm:maincross} tell us that $Y_\ul$, the ultraproduct of $\{Y_n\}_n$, is an external cross section with respect to some map $\Psi_\ul\colon G\times Y_\ul\rightarrow X_\ul$. Theorem \ref{thm:mainrel} implies that there exists a compact subset $K\supseteq B_0$ such that the ultraproduct of the sequence of graphed equivalence relations $(\rel_n^{K},\Theta^K_n)$ is the cross equivalence relation $\rel_\ul$ on $Y_\ul$ which is graphed by $\Theta^K_\ul$. The graphings $\Theta^K_n$ are $K$-supported as in Definition \ref{dfn:kgraph}, that is the equivalence relation $\rel_n^K$ is generated by the element of $K$ which maps points of $Y_n$ to $Y_n$.

For $y\in Y_n$, put 
\begin{align*}
    \omega_{n,y}^K\coloneqq&\{g\in G\colon (y,gy)\in \rel_n^K\};\\
    \omega_{n,y}\coloneqq&\{g\in G\colon gy\in Y_n\}.
\end{align*}
Remark that if $g\in\omega_{n,y}^K$, then 
$\omega_{n,gy}^K=\omega_{n,y}^Kg^{-1}$ and in particular that $\omega_{n,y}^Ky\Gamma y^{-1}=\omega_{n,y}^K$. Similar equations hold for $\omega_{n,y}$ and moreover $\omega_{n,y}=\cup_{g\in\omega_{n,y}}\omega_{n,gy}^Kg$.

For every $n\in\mathbb N$ and $y\in Y_n$, let us set 
\begin{align*}
    \co_{n,y}^K\coloneqq&\{ O_ifg\colon 1\leq i\leq m, f\in F\text{ and }g\in\omega_{n,y}^K\};\\
    \co_{n,y}\coloneqq&\{ O_ifg\colon 1\leq i\leq m, f\in F\text{ and }g\in\omega_{n,y}\}.
\end{align*}

 Observe that $\co_{n,y}^K$ is a good cover of $M_{n,y}^K\coloneqq \cup_{i}O_iF\omega_{n,y}^K$ and $\co_{n,y}$ is a good cover of $M_{n,y}\coloneqq \cup_{i}O_iF\omega_{n,y}$. Clearly $\omega_{n,y}^K\subseteq \omega_{n,y}$ so that $M_{n,y}^K\subseteq M_{n,y}$ and 
 \[M_{n,y}=\bigcup_{i}O_iF\omega_{n,y}=\bigcup_i \bigcup_{g\in\omega_{n,y}}O_iF\omega_{n,gy}^Kg=
 \bigcup_{g\in \omega_{n,y}}M_{n,gy}^Kg.\]
 
 \begin{lem}
  For every $y\in Y_n$ and $g\in G$, 
  \begin{itemize}
      \item if $g\in\omega_{n,y}^K$, then $M_{n,gy}^Kg=M_{n,y}^K$;
      \item if $g\in\omega_{n,y}\setminus \omega_{n,y}^K$, then $M_{n,gy}^Kg\cap M_{n,y}^K=\emptyset.$ 
  \end{itemize}
 \end{lem}
 \begin{proof}
   Suppose that $g\in\omega_{n,y}^K$. Then $\omega_{n,y}^K=\omega_{n,gy}^Kg$ and hence \[M_{n,y}^K=\cup_{i}O_iF\omega_{n,y}^K=\cup_{i}O_iF\omega_{n,gy}^Kg=M_{n,gy}^Kg.\]
   
   Take $y\in Y_n$ and $g\in G$. Assume that  \[(\cup_i O_i)F\omega_{n,y}^{K}\bigcap (\cup_i O_i)F\omega_{n,gy}^{K}g\neq\emptyset.\]
   Then there are $h_y\in \omega_{n,y}^{K}$ and $h_{gy}\in \omega_{n,gy}^{K}$ such that \[(\cup_i O_i)Fh_yg^{-1}h_{gy}^{-1}\bigcap (\cup_i O_i)F\neq \emptyset\] and therefore $h_yg^{-1}h_{gy}^{-1}\in B_0\subseteq K$. Fix $k\in K$ such that $h_y=kh_{gy}g$. Observe that $(y,kh_{gy}gy)=(y,h_yy)\in\rel_{n}^K$. 
   Since both $kh_{gy}gy$ and $h_{gy}gy$ are elements of $Y_n$, by definition of $\rel_n^K$ we have that $(kh_{gy}gy,h_{gy}gy)\in\rel_n^K$. Therefore $(y,h_{gy}gy)\in\rel_n^K$. By hypothesis $(gy,h_{gy}gy)\in \rel_n^K$, thus $(y,gy)\in\rel_n^K$ and we have concluded the proof of the lemma. 
 \end{proof}
 
 Remark that the subsets $M_{n,y}^K$ are not necessarily connected, they are unions of connected components which are close with respect to $K$. We will now consider the quotients $M_{n,y}/(y\Gamma_n y^{-1})$. We will start by showing the analogous of Lemma \ref{lem: invariant free Gamma cover}.

 \begin{lem}
  For every $y\in Y_n$, the cover induced by $\co_{n,y}$ on $M_{n,y}/(y\Gamma_n y^{-1})$ is a good cover.
 \end{lem}
 \begin{proof}
 The proof is similar to the proof of Lemma \ref{lem: invariant free Gamma cover}.  In order to show that the induced cover is a good cover, it is enough to show that given any $O,O'\in\co_{n,y}$ there is at most an element $\gamma\in y\Gamma_n y^{-1}$ such that the intersection $O\cap O'\gamma$ is not empty. Indeed if it is the case, then the intersection of translates of $\co_{n,y}$ in the quotient $M_{n,y}/y\Gamma_ny^{-1}$ are homeomorphic to the corresponding intersection in $M_{n,y}$. Since the latter are either empty or contractible, the former have to satisfy the same property. That is, the cover induced by $\co_{n,y}$ on $M_{n,y}/(y\Gamma_n y^{-1})$ is a good cover.
 
 So let us fix such open sets and by definition $O=O_if_1g_1$ and $O'=O_jf_2g_2$ for some $1\leq i,j\leq m$, $f_1,f_2\in F$ and $g_1,g_2\in\omega_{n,y}$. Assume that there are $\gamma_1,\gamma_2\in \Gamma$ such that 
 \[O_1f_1g_1\gamma_1\cap O_2f_2g_2\neq\emptyset \text{ and }O_1f_1g_1\gamma_2\cap O_2f_2g_2\neq\emptyset.\]
 Then by the assumptions on the $\{O_i\}_i$, we must have that \[f_1g_1\gamma_1g_2^{-1}f_2^{-1}(f_1g_1\gamma_2g_2^{-1}f_2^{-1})^{-1}=f_1g_1\gamma_1\gamma_2^{-1}g_1^{-1}f_1^{-1}\in WK'\]
 for some compact subgroup $K'$ contained in $A_G$. On the other hand \[f_1g_1\gamma_1\gamma_2^{-1}g_1^{-1}f_1^{-1}\in f_1g_1y\Gamma_n(f_1g_1y)^{-1}.\]
 Since by hypothesis $g_1y\in Y_n$, we have that $f_1g_1y\in FY_n\subseteq (X_n)_V$. That is, we have that \[f_1g_1y\Gamma_n(f_1g_1y)^{-1}\cap V=\{\id\}.\]
 So Lemma \ref{lem:Pdisc}, implies that \[f_1g_1\gamma_1\gamma_2^{-1}g_1^{-1}f_1^{-1}\in f_1g_1y\Gamma_n(f_1g_1y)^{-1}\cap WK'\]
 must be a torsion element. Since $\Gamma_n$ is torsion free, $\gamma_1=\gamma_2$ as claimed. 
 \end{proof}
 
  For every $y\in Y_n$, we have a topological space $M^K_{n,y}/(y\Gamma_ny^{-1})$ and a cover of it $O_{n,y}^K/(y\Gamma_n y^{-1})$. Denote by $\Sigma_{n,y}^KM$ the nerve of this cover. The set $Y_n$ is finite for every $n$, so $\mathcal O_{n,y}/(y\Gamma_n y^{-1})$ is a finite cover for every $y$. Hence $\Sigma_{n,y}^KM$ is a finite simplicial space.
  Moreover given $g\in \omega_{n,y}^K$, the multiplication on the right by $g^{-1}$ induces an homeomorphism \[\alpha_g\colon M^K_{n,y}=\cup_iO_iF\omega_{n,y}^K\rightarrow M^K_{n,gy}=\cup_iO_iF\omega_{n,gy}^K=\cup_iO_iF\omega_{n,y}^Kg^{-1}.\]
 Clearly $\alpha_g$ is equivariant with respect to the action of $y\Gamma_n y^{-1}$ on $M^K_{n,y}$ and the action of $gy\Gamma_n y^{-1}g^{-1}$ on $M^K_{n,gy}$. Therefore $\alpha_g$ induces an isomorphism of simplicial spaces \[\tilde \alpha_g\colon \Sigma_{n,y}^KM\rightarrow \Sigma_{n,gy}^KM.\]
 It is easy to see then that $\Sigma^K_nM\coloneqq \{\Sigma_{n,y}^KM\}_{y\in Y_n}$ is a $\rel_n^K$-simplicial complex. We can therefore argue as in \cite[Proposition 2.3]{ber2004}.
 
\begin{lem}\label{lem: bergeron gaboriau lemma}
  In the above notation, for every $y\in Y_n$ \[\beta_i(\Sigma_n^KM)=\frac{b_i(M_{n,y}/y\Gamma_ny^{-1})}{|Y_n|}.\]
\end{lem}

Before proving the lemma, let us remark that the same arguments as before, show that if $y,gy\in Y_n$, then the multiplication by $g^{-1}$ on the right induces an homeomorphism between  $M_{n,y}/y\Gamma_ny^{-1}$ and $M_{n,gy}/gy\Gamma_ny^{-1}g^{-1}$. Therefore the statement of the lemma does not depend on $y$. 

\begin{proof}
 Fix $n\in\mathbb N$.
 Choose a representative $\{y_l\}_l$ for each $\rel_n^K$-class in $Y_n$ and let $C_l$ be the size of the class of $y_l$. Clearly $\sum_l C_l=|Y_n|$ and we can write $M_{n,y}$ as a disjoint union \[M_{n,y}=\cup_l M^K_{n,y_l}y_l.\]
 Following  \cite[Proposition 2.3]{ber2004} we obtain that
 \[\beta_i(\Sigma_n^{K} M)=\sum_{l} \frac{|C_l|}{|Y_n|}\frac{b_i\left(M_{y_l}^{K}\bigr/y_l\Gamma_n y_l^{-1}\right)}{|C_l|}=\sum_{l} \frac{b_i\left(M_{y_l}^{K}y_l\bigr/y\Gamma_ny^{-1}\right)}{|Y_n|}=\frac{b_i(M_{n,y}/y\Gamma_ny^{-1})}{|Y_n|}.\qedhere\]
 \end{proof}
 
 We finally prove the main technical hearth of Theorem \ref{thm:maintechnical}.
 
 \begin{prop}\label{prop: limit l2 betti numbers}
  In the above notation, \[\covol(Y_\ul)\beta_i(G)=\lim_\ul\beta_i(\Sigma_n^KM).\]
 \end{prop}
 
 The proposition will follow from Theorem \ref{thm:mainrel} and \cite[Theorem 5.9]{CGS}. For the proof, we will need to work with ultraproducts of simplicial complexes, which we will recall only in our specific setting. 
 
 \begin{proof}
  Consider our fixed cross sections $Y_n\subseteq X_n$ and equip them with the renormalized counting measure $\nu_n$. On each finite set $Y_n$, we have the equivalence relation $\rel_n^K$ graphed by $\Theta_n^K$. The graphings $\Theta_n^K$ are $K$-supported as in Definition \ref{dfn:kgraph}. Recall that given $y,y'\in Y_n$, then there is $\ph\in\Theta_n^K$ such that $\ph(y)=y'$ if and only if $y'\in Ky$. By Theorem \ref{thm:mainrel}, the ultraproduct of the graphed equivalence relations $(\rel_n^K,\Theta^K_n)$ is the equivalence relation $\rel_\ul$ which is the cross-equivalence relation of the external cross section $Y_\ul=[Y_n]_\ul$ for the action of $G$ on the regular ultraproduct $X_\ul^R$.
  
  For each $n$, we also have a $\rel_n^K$-simplicial complex $\Sigma_n^K M$. We will now take the \textit{ultraproduct} of this sequence of simplicial complexes to obtain a $\rel_\ul$-simplicial complex. Let us start with the $0$-skeleton. 
  We can define a \text{length function} $\ell_n$ on the $0$-simplices of $\Sigma_n^K M$ which associate to $O_ifg\times \{y\}$ the minimal $n$ such that there is $h\in (K\cup K^{-1})^n$ satisfying \[(O_ifg\times \{y\})h=O_ifgh^{-1}\times \{hy\}\in \{O_if\}_{1\leq i\leq m,f\in F}\times Y_n.\]
  
  The $0$-simplices of $\Sigma_{\ul,y_\ul}^K M$ will be given by $\ul$-equivalence classes of sequences of $0$-simplices $[O_{i_n}f_ng_n]_\ul$ where $O_{i_n}f_ng_n$ is a $0$-simplex of $\Sigma_n^K M$, that is $1\leq i_n\leq m$, $f_n\in F$ and $g_n\in\omega_{n,y_n}^K$, with the additional assumption that 
  \[\lim_\ul\ell_n(O_{i_n}f_ng_n\times\{y_n\})<\infty.\]
  Remark that since $m$ and $F$ are finite, one can always assume that $i_n=i\leq m$ and $f_n\in F$ are fixed. Since the length $\ell_n(O_if\times \{y\})=0$, the $0$-simplices of $\Sigma_{\ul,[y_n]_\ul}$ are therefore given by $[O_ifg_n]_\ul$ where $g_n\in \omega_{n,y_n}^K$ is any bounded sequence. 
  Define $g_\ul\coloneqq \lim_\ul g_n$. Now remark that by definition of the regular ultraproduct, we must have that $[g_ny_n]_\ul=g_\ul[y_n]_\ul$ (inside $X_\ul$ and hence in $Y_\ul$ since $\Psi_\ul$ is injective). That is, if we denote \[\omega_{[y_n]_\ul}\coloneqq\{g\in G\colon ([y_n]_\ul,g[y_n]_\ul)\in\rel_\ul\}\] then the $0$-simplices of $\Sigma_{\ul,y_\ul}^K M$ are given by $O_ifg$ with $1\leq i\leq m$, $f\in F$ and $g\in\omega_{[y_n]_\ul}$. 

 The higher skeletons of $\Sigma_{\ul,y_\ul}^KM$ are constructed similarly. That is a $k$-simplex of $\Sigma_{\ul,y_\ul}^KM$ corresponds to a $k+1$-tuple 
 \[[O_{i_0}f_{i_0}g^{(0)}_n]_\ul,\ldots,[O_{i_k}f_{i_k}g^{(k)}_n]_\ul\] such that $i_0,\ldots,i_k\in\{1,\ldots,m\}$, $f_0,\ldots, f_k\in F$ and $g^{(0)}_n,\ldots,g^{(k)}_n\in\omega_{n,y_n}^K$ are such that \begin{equation}\label{eq: limit simplex}
 \bigcap_{j=0,\ldots,k}O_{i_j}f_{i_j}g^{(j)}_n\neq \emptyset\text{ for }\ul\text{-almost every }n. 
 \end{equation}
 
 It is easy to see that $\Sigma_{\ul}^KM=\{\Sigma_{\ul,y_\ul}^KM\}_{y_\ul\in Y_\ul}$ is a $\rel_\ul$-simplicial complex \cite[Proposition 5.6]{CGS} and it is the \textit{ultraproduct} of the simplicial complexes $\Sigma_n^KM$. Since the action of $G$ is proper, the simplicial complexes $\{\Sigma_n^KM\}$ are uniformly bounded and hence we can apply \cite[Theorem 5.9]{CGS} to get that for every $i$, \[\beta_i(\Sigma_\ul^KM)=\lim_\ul\beta_i(\Sigma_n^KM).\]
 
 We are left to show that $\beta_i(\Sigma_\ul^KM)$ is contractible. Indeed in this case, we would have that \[\lim_\ul\beta_i(\Sigma_n^KM)=\beta_i(\Sigma_\ul^KM)=\beta_i(\rel_\ul)=\beta_i(G)\covol(Y_\ul).\]
 
 Before showing that this simplicial complex is contractible, let us give a different description on what happens over each point $y_\ul$. For this, let us fix $[y_n]_\ul\in Y_\ul$. For every $k\in\mathbb N$, let us denote by $\Sigma_{n,y_n}^{K,(k)}M\subseteq \Sigma_{n,y_n}^{K}M$ the finite simplicial sub-complex consisting on all the simplices of $\Sigma_{n,y_n}^{K}M$ whose $0$-skeleton has $\ell_n$-length less than $k$. Since the action of $G$ is cocompact and $Y_n$ is $V$-discrete, the number of simplices on each dimension of $\Sigma_{n,y_n}^{K,(k)}M$ can be bounded independently of $n$. Therefore there are only finitely many possibility for the simplicial complexes  $\Sigma_{n,y_n}^{K,(k)}M$. The ultrafilter $\ul$ select one of these choices and this finite simplicial complex is exactly the corresponding $\Sigma_{\ul,[y_n]_\ul}^{K,(k)}M$. 
 
 In order to show that the simplicial complex is contractible, we can argue as in \cite[Proposition 5.8]{CGS}: it is enough to show that for every $k_0$ there is $k_1$ such that any homotopy sphere lying in $\Sigma_{\ul,[y_n]_\ul}^{K,(k_0)}M$ is null-homotopic in $\Sigma_{\ul,[y_n]_\ul}^{K,(k_1)}M$. The previous discussion implies that we can prove that for $\ul$-almost every $n$ any homotopy sphere lying in $\Sigma_{n,y_n}^{K,(k_0)}M$ is null-homotopic in $\Sigma_{n,y_n}^{K,(k_1)}M$. This follows from the fact that $M$ is contractible and $\{\Gamma_n\}_n$ is a Farber sequence. Indeed, by definition $\Sigma_{n,y_n}^{K}M$ is the nerve of the cover $\co_{n,y_n}/y_n\Gamma_ny^{-1}_n$. Since $\{\Gamma_n\}_n$ is Farber, for almost every $[y_n]_\ul$, the nerves of the covers $\co_{n,y_n}/y_n\Gamma_ny^{-1}_n$ and $\co_{n,y_n}$ coincide on bigger and bigger subsets. In particular for each $k\in\mathbb N$, the simplicial complex $\Sigma_{n,y_n}^{K,(k)}M$ is the nerve of a finite sub-covering of $\co_{n,y_n}$ of $M$. Since $M$ is contractible, for every $k_0$ there is $k_1$ such that every homotopy sphere in $\Sigma_{n,y_n}^{K,(k_0)}M$ has to be null-homotopic in some $\Sigma_{n,y_n}^{K,(k_1)}M$. Therefore the proof of the proposition is concluded. 
 \end{proof}
 
 We will now show that the \textit{boundary} is small. We say that an open set $O\in \mathcal O_{n,y}/(y\Gamma_n y^{-1})$ is a \textit{boundary set} if the closure of $O$ meets the complement of $M_{n,y}/(y\Gamma_n y^{-1})$ in $M/y\Gamma_ny^{-1}$. Recall that given a neighborhood of the identity $U\subseteq G$ and $n\in\mathbb N$, we defined \[(G)^{\Gamma_n}_{U}=\{g\in G\colon g\Gamma_ng^{-1}\cap U=\{\id_G\}\}. \]
 
 Remark that given a neighborhood of the identity $U\subseteq G$ and any subset $B\subseteq G$, for every $n$ we have 
 \begin{equation}\label{eq: moving neighb}
 B(G)^{\Gamma_n}_{B^{-1}UB}\subseteq (G)^{\Gamma_n}_{U}.
 \end{equation}
 Indeed if $g\in (G)^{\Gamma_n}_{B^{-1}UB}$ and $b\in B$, then \[bg\Gamma_ng^{-1}b^{-1}\cap U\subseteq g\Gamma_ng^{-1}\cap B^{-1}UB\subseteq \{\id_G\}.\]
 
\begin{lem}\label{lem:sublinear boundary}
 For every $n$, pick $y_n\in Y_n$. Then the number of boundary elements in $\mathcal O_{n,y_n}/(y_n\Gamma y^{-1}_n)$ grows sublinearly with respect to $|Y_n|$ as $n$ tends to infinity (along the ultrafilter $\ul$).
 \end{lem}
 \begin{proof}
 As remarked after Lemma \ref{lem: bergeron gaboriau lemma}, the lemma does not depend on the choice of $y_n\in Y_n$. So let us assume that $y_n$ is the class of the identity for every $n$. Let us fix $g\in\omega_{n,y_n}$ and a neighborhood of the identity $B\subseteq G$. Then $g\in (G)^{\Gamma_n}_B$ if and only if $g\Gamma_n\in (X_n)_B$. Lemma \ref{lem: cross-section} yields that $V^2Y_n\supseteq (X_n)_{V_1}$. We claim that $V^2\omega_{n,y_n}\supseteq (G)_{V_1}^{\Gamma_n}$. Indeed, if $g\in (G)_{V_1}^{\Gamma_n}$, then $g\Gamma_n\in (X_n)_{V_1}$ and therefore there are $v\in V^2$, $g_n\in \omega_n$ such that $g\Gamma_n=vg_ny_n$. That is, $g\in vg_n\Gamma_n\subseteq V^2\omega_{n,y_n}$. In particular, we have that \[M_{n,y_n}=\cup_i O_iF\omega_{n,y_n}\supseteq \cup_iO_iV^2\omega_{n,y_n}\supseteq \cup_iO_i(G)_{V_1}^{\Gamma_n}.\]
 If $g\in\omega_{n,y_n}\cap (G)_{(SF)^{-1}V_1SF}^{\Gamma_n}$, Equation \eqref{eq: moving neighb} yields
 \[\cup_i\bar O_iFg\subseteq \cup_iO_iSFg\subseteq \cup_i O_iSF (G)_{(SF)^{-1}V_1SF}^{\Gamma_n}\subseteq \cup_i O_i (G)_{V_1}^{\Gamma_n}\subseteq M_{n,y}.
 \]
 Therefore the number of boundary elements is bounded by \[m|F|\left|Y_n\setminus (X_n)_{(SF)^{-1}V_1SF}\right|.\]
Finally remark that since $\{\Gamma_n\}_n$ is Farber, \[\lim_\ul\frac{\left|Y_n\cap (X_n)_{(SF)^{-1}V_1SF}\right|}{|Y_n|}=1\]
so the proof of the lemma is concluded.
 \end{proof}
 
 As a direct consequence of the above lemma and the Mayer-Vietoris long exact sequence we obtain the following. 
 
 \begin{lem}\label{lem: betti is nabla}
   In the above notation, for every $y\in Y_n$, \[\lim_\ul\frac{b_i(M_{n,y}/y\Gamma_ny^{-1})}{|Y_n|}=\lim_\ul\frac{\nabla_i(M_{n,y}/y\Gamma_ny^{-1},M/y\Gamma_ny^{-1})}{|Y_n|}.\]
 \end{lem}
 \begin{proof}
   Fix $i,n\in\mathbb N$. To simplify our notation, as for the lemma above, we can assume that $y$ is the class of the identity, so that $y\Gamma_ny^{-1}=\Gamma_n$ and set $M_n\coloneqq M_{n,y}$. 
   Denote by $\alpha_{n,i}$ the dimension of the kernel of the map induced by the inclusion $H_i(M_n/\Gamma_n)\rightarrow H_i(M/\Gamma_n)$. By definition \[\alpha_{n,i}+\nabla_i(M_n/\Gamma_n,M/\Gamma_n)=b_i(M_n/\Gamma_n).\]
   Denote by $E_n\subseteq M_n/\Gamma_n$ the union of the boundary open sets. Set \[D_0\coloneqq M_n/\Gamma_n,\quad D_1\coloneqq (M/\Gamma_n)\setminus \left((M_n/\Gamma_n)\setminus E_n\right).\]
   Clearly $D_0\cap D_1=E_n$. The Mayer-Vietoris long exact sequence of the decomposition $M/\Gamma_n=D_0\cup D_1$ tells us that
   \[H_i(E_n)\rightarrow H_i(M_n/\Gamma_n)\oplus H_i(D_1)\rightarrow H_i(M/\Gamma_n)\] is exact for every $i$ and $n$. Therefore 
   $\alpha_{n,i}\leq b_i(E_n)$. By Lemma \ref{lem:sublinear boundary}, we have that $b_i(E_n)$ grows sublinearly and hence 
   \begin{align*}
       \lim_\ul\frac{\nabla_i(M_n/\Gamma_n,M/\Gamma_n)}{|Y_n|}\leq & \lim_\ul\frac{b_i(M_n/\Gamma_n)}{|Y_n|}\\
       &=\lim_\ul\frac{\alpha_{n,i}+\nabla_i(M_n/\Gamma_n,M/\Gamma_n)}{|Y_n|} \\
       &\leq \lim_\ul\frac{b_i(E_n)+\nabla_i(M_n/\Gamma_n,M/\Gamma_n)}{|Y_n|}\\
       &=\lim_\ul\frac{\nabla_i(M_n/\Gamma_n,M/\Gamma_n)}{|Y_n|}.\qedhere
   \end{align*}
 \end{proof}

Let us now compare the obtained $M_{n,y}$ with the topological spaces in the statement of Theorem \ref{thm:maintechnical}. 

\begin{lem}\label{lem: sandwich }
 If $y\in Y_n$ is the class of the identity, we have that \[A_M(G)^{\Gamma_n}_{K_0^{-1}V_1K_0}\subseteq M_{n,y}^K\subseteq A_M(G)_{V}^{\Gamma_n}.\]
\end{lem}
\begin{proof}
 By hypothesis we have $FY_n\subseteq (X_n)_{V}$. This can be restated as saying that $F\omega_{n,y}\subseteq (G)^{\Gamma_n}_{V}$. Therefore \[M_{n,y}^K=\cup_iO_iF\omega_{n,y}\subseteq \cup_iO_i(G)^{\Gamma_n}_{V}\subseteq A_M(G)_{V}^{\Gamma_n}.\]
 On the other hand, $\cup_iO_iK_0\supseteq A_M$ 
 and $V^2\omega_{n,y}\supseteq (G)_{V_1}^{\Gamma_n}$. Therefore, using Equation \eqref{eq: moving neighb}, we obtain \[ M_{n,y}^K=\cup_iO_iF\omega_{n,y}\supseteq \cup_iO_iV^2\omega_{n,y}\supseteq \cup_iO_i(G)_{V_1}^{\Gamma_n}\supseteq \cup_iO_iK_0(G)_{K_0^{-1}V_1K_0}^{\Gamma_n}\supseteq A_M(G)_{K_0^{-1}V_1K_0}^{\Gamma_n}.\qedhere\]
\end{proof}

Summing up all the previous statements, we obtain the following.

\begin{lem}\label{lem: the last lemma}
  For every symmetric neighborhood of the identity $V$, set $V_2\coloneqq K_0^{-1}V^{-2}K_0^{-1}VK_0V^2K_0$. Then
  we have
  \[\lim_\ul\frac{\nabla_i(A_M(G)^{\Gamma_n}_{V_2} /\Gamma_n,M/\Gamma_n)}{|Y_n|}\leq \beta_i(G)\covol(Y_\ul)\leq \lim_\ul\frac{\nabla_i(A_M(G)^{\Gamma_n}_{V} /\Gamma_n,M/\Gamma_n)}{|Y_n|}.\]
\end{lem}
\begin{proof}
  By assumption, the class of the identity is in $Y_n$. Denote this class by $y_n$. Then
  combining Lemma \ref{lem: bergeron gaboriau lemma}, Proposition \ref{prop: limit l2 betti numbers} and Lemma \ref{lem: betti is nabla}, we obtain \[\beta_i(G)\covol(Y_\ul)=\lim_\ul\frac{\nabla_i(M_{n,y_n}/\Gamma_n,M/\Gamma_n)}{|Y_n|}.\]
  Therefore the lemma follows from Lemma \ref{lem: sandwich }.
\end{proof}

 Let us conclude the proof of Theorem \ref{thm:maintechnical}. Let us fix a neighborhood of the identity $V$. Consider a neighborhood of the identity $U$ such that \[U_2\coloneqq K_0^{-1}U^{-2}K_0^{-1}UK_0U^2K_0\subseteq V.\]
 Then clearly $(G)_{U_2}^{\Gamma_n}\subseteq (G)_V^{\Gamma_n}$. So applying Lemma \ref{lem: the last lemma} to $V$ and $U$, we get 
 \[\beta_i(G)\covol(Y_\ul)\leq \lim_\ul\frac{\nabla_i(A_M(G)^{\Gamma_n}_{V} /\Gamma_n,M/\Gamma_n)}{|Y_n|}\leq \lim_\ul\frac{\nabla_i(A_M(G)^{\Gamma_n}_{U_2} /\Gamma_n,M/\Gamma_n)}{|Y_n|}\leq \beta_i(G)\covol(Y_\ul).\]
 
 Since $\covol(Y_\ul)=\lim_n\covol(Y_n)$ (Theorem \ref{thm:maincross}) and $\covol(Y_n)=\covol(\Gamma_n)/|Y_n|$ (Lemma \ref{lem:filling}), we get 
  \[\beta_i(G)=\lim_\ul\frac{|Y_n|}{\covol(\Gamma_n)}\lim_\ul\frac{\nabla_i(A_M(G)^{\Gamma_n}_{V} /\Gamma_n,M/\Gamma_n)}{|Y_n|}=\lim_\ul\frac{\nabla_i(A_M(G)^{\Gamma_n}_{V} /\Gamma_n,M/\Gamma_n)}{\covol(\Gamma_n)}\]
and the proof of Theorem \ref{thm:maintechnical} is completed.

\subsection{Examples of well-covered spaces}\label{subsection:examples of well-covered}

We will now provide some examples of well-covered spaces and in particular we will finish the proof of Corollary \ref{crl:excomp}. All our examples will be $G$-CW complexes, see \cite{luc2005}.

\begin{lem}\label{lem: riemann well-covered}
  Let $G$ be a \lcsc group and suppose it acts on the Riemannian manifold $M$ in a cocompact, isometric and proper way.
  Then $M$ is a well-covered $G$-space.
\end{lem}
\begin{proof}
  Let us denote by $d$ the geodesic distance on $M$. 
  Since small geodesic balls of a Riemannian manifold are convex, they form a good cover of the space, see \cite{doc1992}. We need therefore to show only the last point of Definition \ref{dfn:wellcovered}. Consider a compact subset $C\subseteq M$ such that $CG=M$. Set $A_M\coloneqq B_1^M(C)$. Since the action is proper, the set \[A_G\coloneqq \{g\in G\colon gC\cap C\neq \emptyset\}\]
  is compact. Let us denote by $K_A\subseteq A_G$ the union of all compact subgroups contained in $A_G$. Let us fix an open neighborhood of the identity $W\subseteq G$. Remark that $WK_A$ is open.
  For every $x\in C$ let us define \[\eps(x)\coloneqq d(x,x(G\setminus (WK_A))).\]
  Remark that if $xg=x$, then $g$ is in a subgroup contained in $A_G$, that is $g\in K_A$. Moreover since the action is cocompact, and in particular since $M/G$ is Hausdorff, each orbit is closed and the map $g\mapsto xg$ is closed. Therefore, $x(G\setminus (WK_A))$ is a closed subset not containing $x$ and the function $x\mapsto d(x,x(G\setminus (WK_A)))$ is continuous. In particular, $\eps_0\coloneqq \min(\min_{x\in C}\eps(x),1)$ is strictly positive. 
  
  Since $M/G$ is compact, there are finitely many $p_1,\ldots,p_m\in C$ such that $\cup_i B_{\eps_0/4}^M(p_i)G=M$. Now assume that for some $1\leq i,j\leq m$ and $g_1,g_2\in G$, we have that \[B_{\eps_0/4}^M(p_ig_1)\cap B_{\eps_0/4}^M(p_j)\neq \emptyset\text{ and }B_{\eps_0/4}^M(p_ig_2)\cap B_{\eps_0/4}^M(p_j)\neq \emptyset.\] 
  Then the triangular inequality gives us that $d(p_ig_1,p_ig_2)<\eps_0$. Since the action is isometric, we obtain that $d(p_ig_1g_2^{-1},p_i)<\eps_0$. Finally the choice of $\eps_0\leq\eps(p_i)$ tells us that $g_1g_2^{-1}\in WK_A$ and the proof of the lemma is completed. 
\end{proof}

\begin{rmk}
  Note that in the proof of Lemma \ref{lem: riemann well-covered} the assumption that the action is isometric is not crucial. Indeed, it is enough to know that for every $\eps>0$ there is $\delta$ such that whenever $x,y\in M$ are such that $d(x,y)<\delta$, then $d(xg,yg)<\eps$ for every $g\in G$.
\end{rmk}

A group $G$ is said to be \textit{almost connected} if the quotient of $G$ by the connected component of the identity is compact. If $G$ is almost connected, then $M\coloneqq K_0\backslash G$ is a $\underline{\mathrm{E}}G$ space for some maximal compact subgroup $K_0$ of $G$, see \cite[Theorem A.5]{abe1974} and \cite[Corollary 4.14]{abe1978} or \cite{luc2005}. Observe that any almost connected group admits a normal compact subgroup such that the quotient is a Lie group \cite[Theorem 8]{glu1960}. Therefore $M$ is naturally a contractible Riemannian manifold on which $G$ acts by isometry and hence $M$ is a contractible well-covered $G$-space. We therefore have the following.

\begin{fact*}
  Every almost connected group acts on a contractible well-covered $G$-space.
\end{fact*}

We now examine our second class of examples.

\begin{lem}\label{lem: simplicial well-covered}
  Suppose that the \lcsc group $G$ acts simplicially, properly and cocompactly on a simplicial complex $M$ (with countably many cells). Then $M$ is a well-covered $G$-space.
\end{lem}

In order to prove Lemma \ref{lem: simplicial well-covered}, we could introduce a metric on $M$ and proceed as in Lemma \ref{lem: riemann well-covered}. We will present instead a combinatorial proof which uses the barycentric subdivision. Let us recall that if $M$ is a simplicial complex, the $0$-simplices of its barycentric subdivision $BS(M)$ are in correspondence with the simplices of $M$. Two $0$-simplices in $BS(M)$ are the boundary of a $1$-simplex in $BS(M)$ if and only if one is contained in the other as simplices of $M$. It is convenient to isolate the following lemma. 

\begin{lem}\label{lem:simplicial game}
  Suppose that the \lcsc group $G$ acts simplicially on a simplicial complex $M$. Suppose that $\sigma_1^1,\sigma_2^1$ are simplices of $BS(M)$ and take $g_1,g_2\in G$ such that $\sigma_1^1$ contains both $\sigma_2^1g_1$ and $\sigma_2^1g_2$. Then $\sigma_2^1g_1=\sigma_2^1g_2$.
\end{lem}
\begin{proof}
  Let $q^1$ be any $0$-simplex of $\sigma_2^1$. Then both $q^1g_1$ and $q^1g_2$ are $0$-simplices of $\sigma_1^1$. If they are distinct, there is a $1$-simplex in $BS(M)$ whose boundary is $\{q^1g_1,q^1g_2\}$. In this case we would have that $q^1g_1$ and $q^1g_2$ are associated with simplices of $M$ of different dimensions, which is impossible since the action of $G$ is simplicial.   
\end{proof}

\begin{proof}[Proof of Lemma \ref{lem: simplicial well-covered}]
  Denote by $BS_2(M)$ the barycentric subdivision of $BS(M)$. We claim that the covering by stars of the $0$-simplices of $BS_2(M)$ satisfies the required conditions. First observe that the intersection of stars is either empty or is a star of a higher dimensional simplex. Therefore it is a good cover and since the action of $G$ is simplicial, it is $G$-invariant. We are left again to prove the last condition. As in Lemma \ref{lem: riemann well-covered}, we consider a compact subset $C\subseteq M$ such that $CG=M$. We let $A_M$ be the closed star around $C$. Since the action is proper, the set \[A_G\coloneqq \{g\in G\colon gA_M\cap A_M\neq \emptyset\}\]
  is compact.
  
  Given a $0$-simplex $p^2$ of $BS_2(M)$, we will denote by $S(p^2)$ the open star around $p^2$. By cocompactness, there are finitely many $p_i^2\in C$ for $i=1,\ldots,m$, such that $\cup_i S(p_i^2)G=M$. Assume now that there are $g_1,g_2\in G$ and $1\leq i,j\leq m$ such that 
  \[S(p_i^2g_1)\cap S(p_j^2)\neq \emptyset\text{ and }S(p_i^2g_2)\cap S(p_j^2)\neq \emptyset.\] 
  For every $1\leq i\leq m$, we denote by $\sigma_i^1$ the simplex of $BS(M)$ associated with $p_i^2$. Since the dimension of $\sigma_i^1g$ is the same for every $g\in G$, we must have that one of the two following conditions hold
  \begin{itemize}
      \item $\sigma^1_j$ contains both $\sigma^1_ig_1$ and $\sigma^1_ig_2$;
      \item $\sigma^1_j$ is contained in both $\sigma^1_ig_1$ and $\sigma^1_ig_2$.
  \end{itemize}
  In the first case, Lemma \ref{lem:simplicial game} implies that $g_1g_2^{-1}$ fixes $\sigma^1_i$. In the second, $\sigma_i^1$ contains both $\sigma_j^1g_1^{-1}$ and $\sigma_j^1g_2^{-1}$ and hence Lemma \ref{lem:simplicial game} implies that $g_2^{-1}g_1$ fixes $\sigma^1_j$. Therefore $g_1g_2^{-1}$ fixes $\sigma_j^1g_2^{-1}$. Remark that $p_i^2\in C$ and that $p_j^2g_2^{-1}$ is connected by a $1$-simplex in $BS_2(M)$ to $p_i^2$ and therefore $p_j^2g_2^{-1}\in A_M$. In either cases, the element $g_1g_2^{-1}$ fixes a $0$-simplex in $A_M$, therefore it is contained in a compact subgroup contained in $A_G$.
\end{proof}

Well-covered $G$-spaces behave nicely with respect to products.

\begin{lem}
 Let $G_1,G_2$ be \lcsc groups and let $M_i$ be well-covered $G_i$-spaces for $i=1,2$. Then $M_1\times M_2$ is a well-covered $G_1\times G_2$-space.
\end{lem}
\begin{proof}
Clearly the action of $G_1\times G_2$ on $M_1\times M_2$ is proper and cocompact. Also remark that if $\mathcal O^i$ is a $G_i$-invariant good cover of $M_i$, then the collection \[\mathcal O^1\times \mathcal O^2\coloneqq\{O^1\times O^2\colon O^i\in \mathcal O^i, i=1,2\}\]
is a $G$-invariant good cover and the last condition can be easily verified. 
\end{proof}

Finally assume that $G$ is a general locally compact group and denote by $G_0\leq G$ the connected component of the identity. 

\begin{fact*}
 If the group $G/G_0$ acts cocompactly and properly on a contractible simplicial complex $M_0$, then there is a contractible well-covered $G$-space.
\end{fact*}

We will not prove this fact. In \cite{luc2000} it is explained how to induce a $\underline{\mathrm{E}}(G/G_0)$ space to a $\underline{\mathrm{E}}G$ without changing the quotient. Following their construction, the preimage of each point in the quotient is a countable union of connected Riemannian manifolds. Therefore we can construct good covers by taking the product of the cover by stars of $M_0$ and small geodesic balls in the manifolds. 


\end{document}